\documentclass[12pt]{article}
\usepackage{amssymb,enumerate,mathrsfs, amsmath, amsthm, bbm}
\usepackage{url,titletoc,enumitem}
\usepackage{tikz-cd}
\usepackage{hyperref}
\pdfoutput=1
\DeclareMathOperator{\sqf}{sqf}
\DeclareMathOperator{\Res}{Res}
\DeclareMathOperator{\Sym}{Sym}
\DeclareMathOperator{\Spec}{Spec}
\DeclareMathOperator{\abs}{abs}

\DeclareMathOperator{\adelic}{adelic}
\DeclareMathOperator{\Cl}{Cl}
\DeclareMathOperator{\triv}{triv}
\DeclareMathOperator{\sq}{sq}

\newtheorem{theorem}{Theorem}[section]
\newtheorem{cor}[theorem]{Corollary}
\newtheorem{lemma}[theorem]{Lemma}

\newtheorem{prop}[theorem]{Proposition}
\newtheorem{definition}[theorem]{Definition}

\newtheorem{lem}[theorem]{Lemma}

\newtheorem*{lemma*}{Lemma}

\numberwithin{equation}{section}

\def\bbR{ {\mathbb R}}
\def\bbN{ {\mathbb N}}

\def\bbQ{ {\mathbb Q}}
\def\bbZ{ {\mathbb Z}}
\def\bbP{ {\mathbb P}}
\def\bbC{ {\mathbb C}}

\def\fkX{ {\mathfrak X}}

\def\cO{ {\mathcal O} }

\usepackage{biblatex}
\begin{document}
\sloppy 

\title{Counting Rational Points on the Stacky $\Sym^2 \bbP^1$}
\author{John Yin\footnote{The author is affiliated with the University of Wisconsin-Madison}}
\maketitle

\begin{abstract}
    We prove the weak form of the generalized Batyrev-Manin-Malle conjecture formulated in \cite{ellenberg2021heights} for the stack $\Sym^2\bbP^1 := (\bbP^1 \times \bbP^1)/S_2$, where the $S_2$ action just permutes the two coordinates. In particular, we show that the diagonal $\Delta \subset \Sym^2\bbP^1$ is an accumulating substack.
\end{abstract}

\tableofcontents

\section{Introduction}


Let $\fkX$ be the stack $\Sym^2 \bbP^1:=(\bbP^1 \times \bbP^1)/S_2$, where $S_2$ is the symmetric group with two elements, acting via permutation of coordinates. In this paper, we will prove the weak form of the generalized Batyrev-Manin-Malle conjecture formulated in \cite{ellenberg2021heights} for $\fkX$. There is also a stacky Batyrev-Manin conjecture formulated in \cite{darda2022batyrevmanin}, which further predicts the existence of a leading constant. We do not prove this stronger conjecture, though our proof does not give any indication that the leading constant does not exist. Other point counts with respect to the stacky height defined in \cite{ellenberg2021heights} has been done. \cite{nasserden_xiao_2020} establishes asymptotic bounds on point counts of $\bbP^1_{\bbQ}$ with $3$ stacky points $\{0,1,\infty\}$ and the residue gerbes all equal to $\mu_2$. \cite{phillips_2022} establishes point counts of weighted projective spaces. \cite{boggess2020counting} establishes asymptotic bounds on point counts on certain modular curves.

We perform this count by considering the data that $\fkX$ parametrizes. That is, for any field $k$, the $k$-points of $\fkX$ correspond to a quadratic \'etale extension, along with an element in that extension. Most of the time, the element in the extension determines the quadratic \'etale extension, unless that element is in $k$. Thus, counting algebraic numbers of degree $2$ over $k$ essentially correspond to counting the $k$-points of $\fkX$. Counting algebraic numbers of a fixed degree and bounded absolute height over $\bbQ$ was first done in \cite{masser_vaaler}. The authors approach this problem by counting the polynomials of a fixed degree whose roots have bounded absolute height, instead of the algebraic numbers themselves. Such polynomials have coefficients varying in a ``nice" region, and they are able to apply the principle of Lipschitz to obtain bounds on the number of such polynomials. 

In this paper, we will carry out essentially the same count, except with respect to the stacky height in \cite{ellenberg2021heights}. The main difficulty is that while Masser and Vaaler dealt with only the absolute height of an algebraic number, the stacky height involves both the absolute height and the discriminant of the number field the algebraic number generates. Thus, parametrizing by polynomials does not lead to nice regions, as the discriminant of the number field generated by a polynomial is a very discontinuous function in terms of the coefficients of the polynomial. As a result, we have to employ different methods, resulting in worse asymptotic bounds than those in \cite{masser_vaaler}. However, these bounds turn out to be enough to prove the weak form of the generalized Batyrev-Manin-Malle conjecture in the context of $\fkX=\Sym^2 \bbP^1$.

Furthermore, in the setting of Batyrev-Manin conjecture, which deals with point counts on a Fano variety $X$, there is a notion of ``accumulating subvariety" $C \subset X$. It turns out that it could happen that ``most" of the points of bounded height lie on this $C$, despite the fact that it is of a lower dimension. In these cases, the Batyrev-Manin conjecture makes predictions of points on $X \setminus C$. In our case, we will show that the diagonal $\Delta \subset \Sym^2 \bbP^1$ is an ``accumulating substack." We will show that point counts on this substack are on the order of $B^2$ in Theorem \ref{thm:diagonal count}, using just elementary number theory. Points off this substack correspond to points in the quadratic \'etale extension $\bbQ^2$, which we call ``split points," or an algebraic number in a quadratic number field, which we call ``non-split points." We will show that the number of split, non-diagonal points is on the order of $B \log B$ in Theorem \ref{thm:off-diagonal 1}. Finally, we will show that the number of non-split, non-diagonal points corresponding to quadratic field extensions is $\Theta(B \log B)$ in Theorem \ref{thm:main}. Theorem \ref{thm:off-diagonal 1} also just uses elementary number theory, while Theorem \ref{thm:main} is the hardest to prove, and is the main goal of this paper.

The approach we to tackle our main theorem is by separating the non-split, non-diagonal points of $\Sym^2 \bbP^1$ we count into two pieces. The non-split points of $\Sym^2\bbP^1(\bbQ)$ now truly do correspond to algebraic numbers $\alpha$ of degree $2$, and so correspond further to just quadratic polynomials up to scaling by taking their minimal polynomial. Furthermore, the heights of elements of $\Sym^2 \bbP^1(\bbQ)$, expressed in terms of their minimal polynomials $ax^2+bx+c$ representing them, is approximately $H(a:b:c)^2 \sqrt{|\sqf(b^2-4ac)|}$, where $H(a:b:c)$ is the usual height on $\bbP^2$ and $\sqf$ denotes the squarefree part. Thus, the two pieces we separate into are a piece containing quadratic polynomials $f$ where $|\sqf(b^2-4ac)|$ is small, and a piece where $|\sqf(b^2-4ac)|$ is big. In the situation that $|\sqf(b^2-4ac)|$ is small, we use L-functions to parametrize quadratic polynomials which generate a quadratic extension $k$ with fixed discriminant, then apply a Tauberian theorem. In the situation that $|\sqf(b^2-4ac)|$ is big, we switch perspectives and notice that $b^2-4ac$ is divisible by a small square, and count quadratic polynomials of bounded height whose discriminant is divisible by a small square.

The methods in this paper does not immediately allow one to prove weak form of the generalized Batyrev-Manin-Malle conjecture for higher $\Sym^n \bbP^1$. This problem is in general difficult. An element of $\Sym^n\bbP^1(\bbQ)$ is a pair $(k,\alpha)$, where $k$ is an \`etale extension of degree $n$, and $\alpha$ is an element of $k$. The height of such a pair $(k,\alpha)$ is $H_{\abs}(\alpha)^{2n}\sqrt{|\Delta_k|},$ where $H_{\abs}$ is the absolute height and $\Delta_k$ is the absolute discriminant of $k$. By \cite{vaaler2013note}, there exists a primitive element $\alpha_k$ of height bounded by $O(|\Delta_k|^{1/2n})$ in $k$. So, the height of the pairs $(k, \alpha_k)$ would be $O(|\Delta_k|^{3/2})$. Now, the weak form of the generalized Batyrev-Manin-Malle conjecture in this case would predict that the number of such pairs of height bounded by $B$ is in particular bounded by $B \log B$, and so there must be at most $B \log B$ extensions of degree $n$ whose discriminant is bounded by $B^{2/3}$. In other words, the number of extensions of bounded degree and discriminant $O(B^{3/2} \log B)$, independent of $n$. Thus, the weak form of the generalized Batyrev-Manin-Malle conjecture for $\Sym^n \bbP^1$ is much stronger than the state of the art bounds \cite{lemke2022upper} on the number of \textit{number fields} of fixed degree $n$ and bounded discriminant $B$, which is currently $O(B^{c\log(n)^2})$.

\textit{Acknowledgements.} We would like to thank Jordan Ellenberg, Jianhui Li, and Simon Marshall for helpful discussions.



\section{Notations and Conventions}
Let $f,g:\bbZ^{\geq 0} \to \bbR$ be functions.
\begin{itemize}
    \item We say $f(n) = O(g(n))$ if there exists $c>0$, $k$ so that $f(n) \leq c g(n)$ for all $n \geq k$. 
    \item We say $f(n) = \Omega(g(n))$ if there exists $c>0$, $k$ so that $f(n) \geq c g(n)$ for all $n \geq k$. 
    \item We say $f(n) = \Theta(g(n))$ if $f(n)=O(g(n))$ and $f(n) = \Omega(g(n))$. 
    \item We will use $f(n)=O_{\epsilon, \delta, \dots}(g(\epsilon, \delta, \dots,n))$ to denote that the constant $c$ is dependent on $\epsilon, \delta, \dots$. 
    \item We will use $\phi(n)$ to denote the Euler phi function, so that $\phi(n) = |(\bbZ/n\bbZ)^\times|$.
    \item We will denote $\omega(n) = \#\{p \text{ prime}:p|n\}$.
    \item We will use $\mu(n)$ to denote the Mobius function, so that $\mu(n) = 0$ if $n$ is not squarefree, and $\mu(n) = (-1)^{\omega(n)}$ otherwise.
    \item We will use $\sqf(n)$ to denote the squarefree part of $n$. More explicitly, this is a multiplicative function so that $\sqf(p^k)=\begin{cases}1 & 2|k\\
    p & 2 \nmid k \end{cases}$ We also define $\sqf(-1)=-1$, so that $\sqf(-n)=-\sqf(n)$.
    \item Define also $\sq(n) = \sqrt{n/\sqf(n)}$. This is an integer by the definition of $\sqf$.
    \item We define $\Res_{s=a}f(s)$ to be the residue of $f$ at $a$. In other words, $\Res_{s=a}f(s) := \lim_{s \to a} (s-a)f(s)$.
    \item We will always denote height functions with letters $h$ and $H$. Whenever we define $h$, we will define $H(x)$ to be $e^{h(x)}$.
\end{itemize}

\section{Overview}

Let $\fkX:=\Sym^2 \bbP^1 = (\bbP^1 \times \bbP^1)/S_2$, where $S_2$ acts on the two coordinates by permuting them. We will study point counts on this stack.

\begin{definition}
A rational point $x \in \fkX(\bbQ)$ corresponds to a $S_2$-equivariant map $f:\Spec(K) \to \bbP^1 \times \bbP^1$, where $K$ is a degree $2$ etale extension of $\bbQ$. If $K = \bbQ \oplus \bbQ$, then we call $x$ a \textbf{split} point. Otherwise, we call $x$ \textbf{non-split}. Furthermore, if the map $\Spec(K) \to \bbP^1 \times \bbP^1$ factors through $\Spec(K) \to \Spec(\bbQ)$, we call $x$ a \textbf{diagonal} point.
\end{definition}


We may describe non-split points $x$ by the $S_2$-equivariant map $f:\Spec(K) \to \bbP^1 \times \bbP^1$, corresponding to an ordered pair $(\alpha,\beta) \in \bbP^1(K) \times \bbP^1(K)$, which are equivariant under the two actions 
\begin{enumerate}
    \item $(\alpha,\beta) \to (\bar \alpha,\bar \beta)$
    \item and $(\alpha,\beta) \to (\beta, \alpha)$.
\end{enumerate} In conclusion, the datum of an non-split point is equal to the datum of a single $\alpha \in \bbP^1(K)$. The same reasoning applies for split points, except given $(a_0,a_1) \in \bbP^1(\bbQ)$, the notation $\overline{(a_0,a_1)}$ is just $(a_1,a_0)$. In this case we obtain that the datum of a split point is equal to the datum of an ordered pair $(a_0,a_1) \in \bbP^1(\bbQ) \times \bbP^1(\bbQ)$.

Let $K$ be a global field. For $x \in \fkX(K)$, we define $h(x)$ to be the height associated to the tangent bundle of $\fkX$. By the calculations in \cite[Section 3.6]{ellenberg2021heights}, we obtain that given a point $x \in \fkX(\bbQ)$, $$h(x)=\begin{cases}
2h_{\abs}(a)+2h_{\abs}(b) & x \in \fkX(\bbQ) \text{ split, associated with }(a,b) \in \bbP^1(\bbQ) \times \bbP^1(\bbQ)\\
4h_{\abs}(\alpha) + \frac{1}{2}\log |\Delta_{\bbQ[\alpha]}| & x \in \fkX(\bbQ) \text{ non-split, associated with } \alpha \in \bbP^1(K)\end{cases},$$ where $h_{\abs}$ is the logarithmic absolute Weil height. 

Now define $H_{\abs}(x)=e^{h_{\abs}(x)}$ to be the exponential Weil height, and also define $H(x)=e^{h(x)}$. Our goal is to count $\fkX(\bbQ)$ with respect to $H(x)$. Counting split and diagonal points is easy, involving only some elementary number theory, and we may obtain the exact asymptotic. However, counting non-split points is more difficult and instead we only obtain an asymptotic upper and lower bound with the same magnitude.

We begin with the main statement of the paper, which shows that the diagonal is an accumulating subset.
\begin{theorem}\label{thm:intuitive main theorem}
    Assuming GRH, we have that $$\{x \in \fkX(\bbQ) : x \text{ diagonal}, H(x) \leq B\}=\Theta(B^2),$$ and $$\{x \in \fkX(\bbQ) : x \text{ off-diagonal}, H(x) \leq B\} = \Theta(B \log B).$$
\end{theorem}
\begin{proof}
    This is just the combination of Theorems \ref{thm:diagonal count}, \ref{thm:off-diagonal 1}, and \ref{thm:main}.
\end{proof}

\begin{theorem}\label{thm:diagonal count}
Let $c=\sum_{x=1}^\infty \frac{\phi(x)}{x^8}$. Then, $$\{x \in \fkX(\bbQ) : x \text{ diagonal}, H(x) \leq B\}=2cB^2+O(B)$$
\end{theorem}
\begin{proof}
Let $x$ be a diagonal point. Then, $x$ is a map $x:\Spec(K) \to \bbP^1 \times \bbP^1$ whose image is a rational point of the form $(r,r)$, where $r \in \bbP^1(\bbQ)$ (here $K$ may be $\bbQ \times \bbQ$). Also let $r=a/b$, where $a,b$ in lowest form. Then, the height $H(x)$ is $H_{\abs}(r)^4\sqrt{|\Delta_K|}$. Thus, we must count the following set $$\{(r,K):[K:\bbQ]=2, \max(a,b)^4|\Delta_K|^{1/2}\leq B\}.$$

This count is $$\sum_{x=1}^{B^{1/4}}2\phi(x) \#\{K: [K:\bbQ]=2, |\Delta_K| \leq B^2/x^8\}.$$ It is well known that $$\#\{K: [K:\bbQ]=2, |\Delta_K| \leq B^2/x^8\}=\frac{1}{\zeta(2)}\frac{B^2}{x^8}+O(\frac{B}{x^4}).$$ Thus, we need to find $$\sum_{x=1}^{B^{1/4}}2\phi(x)\frac{B^2}{x^8}+O(\sum_{x=1}^{B^{1/4}}2\phi(x)\frac{B}{x^4})=A+O(E).$$ 

First we estimate $A$. We have \begin{align*}
    \sum_{x=1}^{B^{1/4}} \frac{\phi(x)}{x^8} 
    &= c-\sum_{x=B^{1/4}}^{\infty} \frac{\phi(x)}{x^8}\\
    &= c+O(\sum_{x=B^{1/4}}^{\infty} \frac{1}{x^7})\\
    &= c+O(B^{-3/2})\\
\end{align*}
Thus, $A=2cB^2+O(\sqrt{B})$.

Now we estimate $E$. $$E=\sum_{x=1}^{B^{1/4}}2\phi(x)\frac{B}{x^4}=O(B).$$ Thus, $A+E = 2cB^2+O(B)$, as desired.
\end{proof}

\begin{theorem} \label{thm:off-diagonal 1}
$$\{x \in \fkX(\bbQ) : x \text{ split and nondiagonal}, H(x) \leq B\} = \frac{1}{\zeta(2)^2}B \log B + O(B)$$
\end{theorem}
\begin{proof}
Here, we are counting pairs of rational points $(r,s)$ so that $H_{\abs}(r)^2H_{\abs}(s)^2 \leq B$. Writing $r=a/b$ and $s=c/d$ where $a/b$ and $c/d$ are fractions in lowest terms, we see that we need to count $$\{(a,b,c,d):(a,b)=1, (c,d)=1, \max(a,b)^2\max(c,d)^2 \leq B\}.$$ This count is \begin{align*}
    \sum_{x=1}^{\sqrt{B}} \sum_{y=1}^{\sqrt{B}/x} \# \{(a,b):(a,b)=1, &\max(a,b)=x\} \# \{(c,d):(c,d)=1, \max(c,d)=y\}\\
    &=4\sum_{x=1}^{\sqrt{B}}\sum_{y=1}^{\sqrt{B}/x} \phi(x)\phi(y)\\
    &=4\sum_{x=1}^{\sqrt{B}}\phi(x) \sum_{y=1}^{\sqrt{B}/x}\phi(y).
\end{align*} The fact that $$\sum_{y=1}^{\sqrt{B}/x} \phi(y) = \frac{1}{2\zeta(2)}(\sqrt{B}/x)^2+O((\sqrt{B}/x) \log (\sqrt{B}/x))$$ is standard. Thus, we get $$4\sum_{x=1}^{\sqrt{B}}\phi(x)\left(\frac{1}{2\zeta(2)}\left(\frac{\sqrt{B}}{x}\right)^2+O\left(\frac{\sqrt{B}}{x} \log (\sqrt{B}/x)\right)\right)$$\begin{equation} \label{eqn:temp7}
    =\frac{2B}{\zeta(2)}\sum_{x=1}^{\sqrt{B}}\frac{\phi(x)}{x^2}+O(\sum_{x=1}^{\sqrt{B}}\phi(x)\frac{\sqrt{B}\log (\sqrt{B}/x)}{x}) = A+O(E),
\end{equation} where we denote $A$ to be the first term and $E$ to be the second term.

First we estimate $A$. Using summation by parts, we have that \begin{align*}
    \sum_{x=1}^{\sqrt{B}} \frac{\phi(x)}{x^2} 
    &= O(\frac{\sum_{x=1}^{\sqrt{B}}\phi(x)}{\sqrt{B}^2})+\sum_{k=1}^{\sqrt{B}-1}(\frac{1}{k^2}-\frac{1}{(k+1)^2})\sum_{x=1}^k \phi(k)\\
    &= O(1)+\sum_{k=1}^{\sqrt{B}-1} \frac{2k+1}{k^2(k+1)^2}(\frac{1}{2\zeta(2)}k^2+O(k \log k)) \\
    &=O(1)+\frac{1}{\zeta(2)}\sum_{k=1}^{\sqrt{B}-1} \frac{1}{k}\\
    &=\frac{1}{2\zeta(2)} \log(B) + O(1)
\end{align*} Thus, $A=\frac{1}{\zeta(2)^2}B \log B + O(B).$ This is our main term. We conclude by estimating $E$, which will be an error term. \begin{align*}
    E&=\sum_{x=1}^{\sqrt{B}}\phi(x)\frac{\sqrt{B}\log (\sqrt{B}/x)}{x}\\
    &<\sum_{x=1}^{\sqrt{B}}\sqrt{B}\log (\sqrt{B}/x)\\
    &=B\log(\sqrt{B})-\sqrt{B}\sum_{x=1}^{\sqrt{B}}\log (x)\\
    &=B\log(\sqrt{B})-\sqrt{B}(\sqrt{B} \log(\sqrt{B})+O(\sqrt{B}))\\
    &=O(B)
\end{align*}

\end{proof}

The rest of the paper will be dedicated to proving the following theorem.
\begin{theorem}
\label{thm:main}
Assuming GRH, $$\{x \in \fkX(\bbQ) : x \text{ non-split and nondiagonal}, H(x) \leq B\} = \Theta(B \log B).$$
\end{theorem}

\begin{proof}
We have by definition that an non-split point is given by an element of a quadratic extension of $\bbQ$, and an element of a quadratic extension may be, up to conjugates, identified with its minimal polynomial. Next, by \cite[Prop 1.6.6, Lemma 1.6.7]{bombieri_gubler_2007}, we have $H_{\abs}(\alpha)^2=\Theta(H_{\abs}([a:b:c]))$ where $\alpha$ is a root of $ax^2+bx+c$. We also have that $|\Delta_{\bbQ[\alpha]}|^{1/2}=\Theta(\sqrt{|\sqf(b^2-4ac)|})$.

Thus, we have the asymptotic bounds \begin{align*}
    |\{x \in \fkX(\bbQ) : x \text{ non-split and nondiagonal}, H(x) \leq B\}| = \Theta(&|\{(a,b,c) \in \bbZ: a>0, \gcd(a,b,c)=1,\\
    &\max(|a|,|b|,|c|)^2\sqrt{|\sqf(b^2-4ac|)|}\leq B,\\
    &\sqf(b^2-4ac) \notin \{0,1\}\}|)
\end{align*}  Thus, it suffices to upper and lower bound the latter set. To do this, we have two technical lemmas, which we prove right after this theorem.

\begin{lem} \label{lem:technical 1}
    We have that $$|\{(a,b,c) \in \bbZ: a>0, \gcd(a,b,c)=1, \max(|a|,|b|,|c|)^2\sqrt{|\sqf(b^2-4ac)|}\leq B, \sqf(b^2-4ac) \notin \{0,1\}\}|$$ $$=O(|\{(a',b,c'): \sqf(a'^2+b^2-c'^2) \notin \{0,1\},\max(|a'|,|b|,|c'|)^2 \sqrt{|\sqf(a'^2+b^2-c'^2)|}\leq B\}|).$$
\end{lem}

\begin{lem} \label{lem:technical 2}
    We have that  \begin{align*}
        &|\{(a,b,c) \in \bbZ: a>0, \gcd(a,b,c)=1, \max(|a|,|b|,|c|)^2\sqrt{|\sqf(b^2-4ac)|}\leq B, \sqf(b^2-4ac) \notin \{0,1\}\}|\\
        &=\Omega(|\{(a,b,c): \gcd(a,b) = 1, \max(|a|,|b|,|c|)^2 \sqrt{x}\leq B, a^2+b^2 \equiv 5 \mod 8,\\
        &\hspace{8cm} \sqf(a^2+b^2-c^2) \equiv b \equiv 0 \mod 2\}|).
    \end{align*}
\end{lem}

Now let $$S:= \{(a',b,c'): \sqf(a'^2+b^2-c'^2) \notin \{0,1\},\max(|a'|,|b|,|c'|)^2 \sqrt{|\sqf(a'^2+b^2-c'^2)|}\leq B\}$$ be the upper-bounding set appearing in Lemma \ref{lem:technical 1}. For this set, we split it up into different pieces, as follows. Define $$S_{x}=\{(a',b,c') \in S: \sqf(a'^2+b^2-c^2)=x\} \text{ and } S_{y}=\{(a',b,c') \in S: \sq(a'^2+b^2-c^2)=y\}.$$ 


We have the following bounds on $S_x$'s and $S_y$'s, all of which are conditioned on GRH. The proof of each of the following three proposition will be its own section. Together, these bounds will be more than enough to give our upper bound on $S$. In particular, we do not need the full power of these propositions.

\begin{prop}
\label{prop:sum_of_Sx_for_x_small}
Assuming GRH, for all $\epsilon > 0$, $$\sum_{\substack{2 \leq x \leq B^{2/3-\epsilon}\\x \text{ squarefree}}} |S_x| = O_{\epsilon}(B \log B).$$
\end{prop} This is done in Section \ref{sec:S_x bound positive}.
\begin{prop}
\label{prop:upper_bound_small_x_negative}
Assuming GRH, for all $\epsilon > 0$,  $$\sum_{\substack{-1 \geq x \geq -B^{2/5-\epsilon}\\x \text{ squarefree}}} |S_x| = O_{\epsilon}(B \log B).$$
\end{prop} This is done in Section \ref{sec:S_x bound negative}.
\begin{prop}
\label{prop:sum_of_Sy_for_y_big}
Assuming GRH, for all $\epsilon > 0$,  $$\sum_{1 \leq y \leq B^{5/19-\epsilon}} |S_y| = O_{\epsilon}(B \log B).$$
\end{prop} This is done in Section \ref{sec:S_y bound}.

With these propositions, we may prove the upper bound. These three propositions give us bounds on $|S_x|$ for $|x| \leq B^{2/5-\epsilon}$ and $|S_y|$ for $|y| \leq B^{5/19-\epsilon}$. Thus, it suffices to show $$S \subset \bigcup_{|x| \leq B^{2/5-\epsilon}} S_x \cup \bigcup_{y \leq  B^{5/19-\epsilon}}S_y$$ for some small $\epsilon$. Indeed, let $(a',b,c')$ be an element of $S$, let $x=\sqf(a'^2+b^2-c'^2)$, and let $y=\sq(a'^2+b^2-c'^2)$. Suppose $|x| \geq B^{2/5-\epsilon}$. Then, as $a'^2+b^2-c'^2=xy^2$ and $\max(|a|,|b|,|c|)^2\sqrt{|x|}\leq B$, we see that $(xy^2)\sqrt{|x|}\leq B$. Thus, $y \leq \frac{B}{|x|^{3/2}} \leq \frac{B^{1/2}}{B^{3/10-3/2\epsilon}}=B^{1/5+3/2\epsilon}$. After choosing an $\epsilon$ so that $1/5+3/2\epsilon \leq 5/19-\epsilon$, we are done with the upper bound.

For the lower bound, define $$T:=\{(a,b,c): \gcd(a,b) = 1, \max(|a|,|b|,|c|)^2 \sqrt{x}\leq B, a^2+b^2 \equiv 5 \mod 8,$$$$ \sqf(a^2+b^2-c^2) \equiv b \equiv 0 \mod 2\}$$ to be the lower-bounding set in Lemma \ref{lem:technical 2}. We prove $|T| \gg B \log B$ in Proposition \ref{prop:thm_main_lower_bound} in Section \ref{sec:lower bound}.

\end{proof}

The combination of these three theorems show that the diagonal $\Delta \subset \Sym^2 \bbP^1$ is an ``accumulating substack." Indeed, on this subvariety, the asymptotic of the point counts are on the order of $B^2$, while the asymptotic of the point counts off the diagonal are on the order of $B \log B$.

We conclude this section with the proof of the two technical Lemmas \ref{lem:technical 1} and \ref{lem:technical 2}.
\begin{proof}[Proof of Lemma \ref{lem:technical 1}]
    We wish to show $$|\{(a,b,c) \in \bbZ: a>0, \gcd(a,b,c)=1, \max(|a|,|b|,|c|)^2\sqrt{|\sqf(b^2-4ac)|}\leq B, \sqf(b^2-4ac) \notin \{0,1\}\}|$$ is asymptotically upper bounded by $$|\{(a',b,c'): \sqf(a'^2+b^2-c'^2) \notin \{0,1\},\max(|a'|,|b|,|c'|)^2 \sqrt{|\sqf(a'^2+b^2-c'^2)|}\leq B\}|.$$ For the upper bound, we may forget the conditions $a>0$ and $\gcd(a,b,c)=1$. Also, we may make a linear change of variables $a = \frac{c'+a'}{2}$, $c = \frac{c'-a'}{2}$, with an inverse change of variables of $a' = a-c$, $c' = a+c$. Then, $b^2-4ac=a'^2+b^2-c'^2$. So there is a bijection between $$\{(a,b,c) : \max(|a|,|b|,|c|)^2 \sqrt{|\sqf(b^2-4ac)|} \leq  B\} $$ and $$ \{(a',b,c') \in \bbZ:\max(|\frac{c'+a'}{2}|,|b|,|\frac{c'-a'}{2}|)^2 \sqrt{|\sqf(a'^2+b^2-c'^2)|}\leq B, a' \equiv c' \mod 2\}.$$ For an upper bound, we may drop the conditions $a' \equiv c' \mod 2$. Furthermore, $$|\{(a',b,c'): \sqf(a'^2+b^2-c'^2) \notin \{0,1\},\max(|a'|,|b|,|c'|)^2 \sqrt{|\sqf(a'^2+b^2-c'^2)|}\leq B\}|.$$\begin{align*}
    \max(|a'|,|b|,|c'|) &= \max(|a-c|,|b|,|a+c|)\\
    &\leq 2 \max(|a|,|b|,|c|) \\
    &= 2\max(|\frac{a'+c'}{2}|,|b|,|\frac{a'-c'}{2}|)
\end{align*} Thus, we have that $$|\{(a,b,c) \in \bbZ: a>0, \gcd(a,b,c)=1, \max(|a|,|b|,|c|)^2\sqrt{|\sqf(b^2-4ac)|}\leq B, \sqf(b^2-4ac) \notin \{0,1\}\}|$$ is bounded by $$|\{(a',b,c'): \sqf(a'^2+b^2-c'^2) \notin \{0,1\},\max(|a'|,|b|,|c'|)^2 \sqrt{|\sqf(a'^2+b^2-c'^2)|}\leq 2B\}|.$$ Since we're only looking for an asymptotic upper bound, we may remove the $2$ in front of the $B$.
\end{proof}

\begin{proof}[Proof of Lemma \ref{lem:technical 2}]
We will show that $$\{(a',b',c'): \gcd(a',b') = 1, \max(|a'|,|b'|,|c'|)^2 \sqrt{\sqf(a'^2+b'^2-c'^2)}\leq B, a'^2+b'^2 \equiv 5 \mod 8, $$$$\sqf(a'^2+b'^2-c'^2) \equiv b' \equiv 0 \mod 2, \sqf(a'^2+b'^2-c'^2) \geq 2\}$$ injects into $$\{(a,b,c) \in \bbZ: a>0, \gcd(a,b,c)=1, \max(|a|,|b|,|c|)^2\sqrt{|\sqf(b^2-4ac)|}\leq B, \sqf(b^2-4ac) \notin \{0,1\}\}$$ via the map $$(a',b',c') \to (\frac{a'-c'}{2}, \frac{b'}{2}, \frac{a'+c'}{2}).$$

Assuming the map is well defined, the injection follows as this is a linear map which is even injective over the reals. So, we just need to show well-definedness. 

We first check the coprime condition. Since $a'^2+b'^2 \equiv 5 \mod 8$, and $b'$ is even, we must have that $a'$ is odd. Thus, $a'^2 \equiv 1 \mod 8$, and so $\frac{b'}{2}$ must be odd, otherwise $b'^2 \equiv 0 \mod 8$ so that $a'^2+b'^2 \not \equiv 5 \mod 8$. We also have that by similar reasoning, $c'$ must be odd. So, we have that $a',c'$ are odd, and $b'$ is even but $\frac{b'}{2}$ is odd. Thus, $\frac{a'-c'}{2}, \frac{b'}{2}, \frac{a'+c'}{2}$ are all integers. Suppose now that these integers are not coprime. Then, there is some prime $p$ dividing all of them. Now since $p|\frac{a'-c'}{2}$ and $\frac{a'+c'}{2}$, we have $p|a'$. Also, since $p | \frac{b'}{2}$, $p|b'$. But this contradicts the assumption that $a',b'$ are coprime.

Now we only need to check the height bound. We have $$\max\left(\left\lvert \frac{a'-c'}{2}\right\lvert, \left\lvert \frac{b'}{2}\right\lvert, \left\lvert \frac{a'+c'}{2}\right\lvert\right) \leq \max(|a'|,|b'|,|c'|)$$ and $$\frac{1}{4}(a'^2+b'^2-c'^2)=(\frac{b'}{2})^2-4(\frac{a'-c'}{2})(\frac{a'+c'}{2}),$$ so that $$|\sqf(a'^2+b'^2-c'^2)|=|\sqf((\frac{b'}{2})^2-4(\frac{a'-c'}{2})(\frac{a'+c'}{2}))|.$$ Thus, \begin{align*}
    B &\geq \max(|a'|,|b'|,|c'|)^2\sqrt{|\sqf(a'^2+b'^2-c'^2)|}\\
    &\geq \max\left(\left\lvert \frac{a'-c'}{2}\right\lvert, \left\lvert \frac{b'}{2}\right\lvert, \left\lvert \frac{a'+c'}{2}\right\lvert\right)^2\sqrt{|\sqf((\frac{b'}{2})^2-4(\frac{a'-c'}{2})(\frac{a'+c'}{2}))|}\\
    &=\max(|a|,|b|,|c|)^2 \sqrt{|\sqf(b^2-4ac)|}
\end{align*}
\end{proof}

\section{Estimate of $|S_x|$ for $1 \leq x \leq B^{2/3-\epsilon}$} \label{sec:S_x bound positive}

In this section, we prove Proposition \ref{prop:sum_of_Sx_for_x_small}. Let $F(n)$ be the number of ways $n$ can be written as a sum of two positive squares, and let \begin{align*}
    G_x(n)&=\#\{\text{ideals } I \subset \cO_{\bbQ[\sqrt{-x}]}: N(I)=n, I \text{ is principal}\} \\
    &= \begin{cases}
    \frac{1}{\omega_x} \#\{(c,y) \in \bbZ^2:c^2+xy^2=n\} & x \equiv 2,3 \mod 4\\
    \frac{1}{\omega_x} \#\{(c,y) \in (\frac{1}{2}\bbZ)^2:c^2+xy^2=n\} & x \equiv 1 \mod 4
\end{cases}
\end{align*} We may identify $S_x$ with $$\{(a,b,c,y):\max(|a|,|b|,|c|)^2 \sqrt{x}\leq B, a^2+b^2=c^2+xy^2\},$$ by just letting $y=\sq(a^2+b^2-c^2)$. Since $x$ is positive, we have $\max(|a|,|b|,|c|)^2\leq B/\sqrt{x}$, and so $a^2+b^2$ is at most $2B/\sqrt{x}$. Then, for a fixed positive $x$, we have $|S_x| \leq 4 \omega_x \sum_{n=1}^{2B/\sqrt{x}}F(n)G_x(n)$. $F(n)$ is a multiplicative function, and $G_x(n)$ is a sum of multiplicative function indexed by the elements of the class group, and so we may use the theory of $L$-functions and Tauberian theorems to get good estimates on $|S_x|$.

Let $$L_x(s)=\sum_{n=1}^{\infty} \frac{F(n)G_x(n)}{n^s}.$$ With $L_x(s)$, we will proceed with the following steps.
\begin{enumerate}
    \item We show that $L_x(s)$ can be meromorphically continued to $\Re(s)>1/2$, with a simple pole at $s=1$.
    \item We find the residue of $L_x(s)$ at $s=1$.
    \item We use the Wiener-Ikehara Tauberian theorem as in Theorem \ref{thm:appendix_wiener_ikehara_with_error} to obtain an estimate on the sum of the coefficients on $L_x(s)$, with an acceptable error term.
\end{enumerate}

We may express $L_x(s)$ as a sum of $L$-functions with Euler product. Let $\widehat{\Cl(\bbQ[\sqrt{-x}])}$ be the characters on the class group of $\bbQ[\sqrt{-x}]$. Let $\chi \in \widehat{\Cl(\bbQ[\sqrt{-x}])}$. Consider the following $L$-functions. Let $$L_0(\chi,s) = \sum_{n=1}^\infty \frac{F(n) \sum_{\text{ideals } I \subset \cO_{\bbQ[\sqrt{-x}]}, N(I)=n} \chi(I)}{n^s}.$$ Thus, $$L_x(s)=\frac{1}{|\Cl(\bbQ[\sqrt{-x}])|}\sum_{\chi \in \widehat{\Cl(\bbQ[\sqrt{-x}])}} L_0(\chi,s).$$

We first make the following definitions. \begin{align*}
    & P_{1,1}=\{\text{primes }p | p \text{ splits in } \bbQ[i] \text{ and } \bbQ[\sqrt{-x}]\} \\
    & P_{1,0}=\{\text{primes }p | p \text{ splits in } \bbQ[i] \text{ and is inert in } \bbQ[\sqrt{-x}]\}\\
    & P_{0,1}=\{\text{primes }p | p \text{ is inert in } \bbQ[i] \text{ and splits in } \bbQ[\sqrt{-x}]\}\\
    & P_{0,0}=\{\text{primes }p | p \text{ is inert in } \bbQ[i] \text{ and } \bbQ[\sqrt{-x}]\}\\
    & P_{1,ram}=\{\text{primes }p | p \text{ splits in } \bbQ[i] \text{ and ramifies in } \bbQ[\sqrt{-x}]\}\\
    & P_{0,ram}=\{\text{primes }p | p \text{ is inert in } \bbQ[i] \text{ and ramifies in } \bbQ[\sqrt{-x}]\}
\end{align*}
Also, we define $$Z_{2,x,\chi}(s)=\begin{cases}
1+\frac{\chi(\pi_2)}{2^s}+\frac{\chi(\bar \pi_2)}{2^s}+\frac{\chi(\pi_2^2)}{2^{2s}}+\frac{\chi(\pi_2 \bar \pi_2)}{2^{2s}}+\frac{\chi(\bar \pi_2^2)}{2^{2s}}+\dots & 2 \text{ splits as } \pi_2 \bar \pi_2 \text{ in } \bbQ[\sqrt{-x}]\\
1+\frac{\chi(\pi_2)}{2^s}+\frac{\chi(\pi_2^2)}{2^{2s}}+\dots & 2 \text{ ramifies in } \bbQ[\sqrt{-x}] \\
1+\frac{\chi(2)}{2^{2s}}+\frac{\chi(4)}{2^{4s}}+\dots & 2 \text{ is inert in } \bbQ[\sqrt{-x}]
\end{cases}$$

Then, we have the Euler product \begin{align}\label{eqn:prod_def_of_Lxs}
L_0(\chi,s)&=Z_{2,x,\chi}(s)\prod_{p \in P_{1,1}}\left(1+\frac{2\chi(\pi)}{p^{s}}+\frac{2\chi(\bar\pi)}{p^{s}}+\frac{3\chi(\pi^2)}{p^{2s}}+\frac{3\chi(\pi\bar \pi)}{p^{2s}}+\frac{3\chi(\bar \pi^2)}{p^{2s}}+\dots\right)\nonumber \\
&\prod_{p \in P_{1,0}}\left(1+\frac{3\chi(p)}{p^{2s}}+\frac{5\chi(p^2)}{p^{4s}}+\dots\right)\prod_{p \in P_{0,1}}\left(1+\frac{\chi(\pi^2)}{p^{2s}}+\frac{\chi(\pi \bar \pi)}{p^{2s}}+\frac{\chi(\bar \pi)^2}{p^{2s}}+\dots\right)\nonumber \\
&\prod_{p \in P_{0,0}}\left(1+\frac{\chi(p)}{p^{2s}}+\frac{\chi(p^2)}{p^{4s}}+\dots\right)\prod_{p \in P_{1,ram}}\left(1+\frac{2\chi(\pi)}{p^{s}}+\frac{3\chi(\pi^2)}{p^{2s}}+\dots\right)\nonumber \\
&\prod_{p \in P_{0,ram}}\left(1+\frac{\chi(p)}{p^{2s}}+\frac{\chi(p^2)}{p^{4s}}+\dots\right).
\end{align}

Let $\chi_x$ be the character associated to the Jacobi Symbol $\chi_x(n)=\left(\frac{x}{n}\right)$. Having obtained that $L_x(s)$ is a sum of Euler products, we get

\begin{lemma}
\label{lem:Lxs_meromorphic_continuation}
$L_x(s)$ may be meromorphically continued to $\Re(s)>1/2$, with a simple pole at $s=1$ of residue $\frac{1}{|\Cl(\bbQ[\sqrt{-x}])|}\Res_{s=1} L_0(\chi_{\triv},s)$.
\end{lemma}
\begin{proof}
From the definition of $L_x(s)$, it suffices to meromorphically continue each $L_0(\chi,s)$. From the product definition of $L_0(\chi,s)$ as in Equation \ref{eqn:prod_def_of_Lxs}, we see that besides the product $$\prod_{p \in P_{1,1}}\left(1+\frac{2\chi(\pi)}{p^{s}}+\frac{2\chi(\bar\pi)}{p^{s}}+\frac{3\chi(\pi^2)}{p^{2s}}+\frac{3\chi(\pi\bar \pi)}{p^{2s}}+\frac{3\chi(\bar \pi^2)}{p^{2s}}+\dots\right),$$ everything else converges absolutely for $\Re(s)>1/2$. So, it suffices to continue this product. We may rewrite this product as $$\prod_{\pi, \chi_{-1}(N(\pi))=1} (1+\frac{\chi(\pi)}{N(\pi)^s})^2$$ up to products of the form $$\prod_p (1+\frac{a_{p,2}}{p^{2s}}+\frac{a_{p,3}}{p^{3s}}+\dots),$$ such that $f(i)=\max_p |a_{p,i}|$ is polynomially bounded independent of $x$. That is, $f(i) \leq P(i)$ for some polynomial $P$, not dependent on $x$. This final product is then again equal to $L(\chi,s)L(\chi \chi_{-1},s)$ (where these are Hecke L-functions) up to products of the form $$\prod_p (1+\frac{a_{p,2}}{p^{2s}}+\frac{a_{p,3}}{p^{3s}}+\dots)$$ with coefficients of bounded growth rate. Now $L(\chi,s)L(\chi \chi_{-1},s)$ is holomorphic for $\chi \neq \chi_{\triv}$, and meromorphic with a simple pole at $s=1$ for $\chi = \chi_{\triv}$. Since the intermediate products of the form $$\prod_p (1+\frac{a_{p,2}}{p^{2s}}+\frac{a_{p,3}}{p^{3s}}+\dots)$$ are absolutely convergent for $\Re(s)>1/2$, we have that $L_0(\chi,s)$ is holomorphic on $\Re(s)>1/2$ for $\chi \neq \chi_{\triv}$ and meromorphic with a simple pole at $s=1$ for $\chi=\chi_{\triv}$. Thus, as $L_x(s)=\frac{1}{|\Cl(\bbQ[\sqrt{-x}])|} \sum_{\chi} L_0(\chi,s)$, we have that $L_x(s)$ is meromorphic with a simple pole only at $s=1$ of residue $\Res_{s=1}L_0(\chi_{\triv},s)$.
\end{proof}

\begin{lemma}
\label{lem:Lxs_residue}
$$\Res_{s=1} L_x(s) = O(\frac{L(\chi_x,1)\prod_{p|x}(1+p^{-1})^2}{\sqrt x})$$
\end{lemma}

\begin{proof}
By Lemma \ref{lem:Lxs_meromorphic_continuation}, we restrict our attention to $\Res_{s=1}L_0(\chi_{\triv},s)$. \begin{align*}
    L_0(\chi_{\triv},s)&=Z_{2,x,\chi_0}(s)\prod_{p \in P_{1,1}}\left(1+\frac{4}{p^{s}}+\frac{9}{p^{2s}}+\dots\right)\prod_{p \in P_{1,0}}\left(1+\frac{3}{p^{2s}}+\frac{5}{p^{4s}}+\dots\right)\\
    &\prod_{p \in P_{0,1}}\left(1+\frac{3}{p^{2s}}+\frac{5}{p^{4s}}+\dots\right)\prod_{p \in P_{0,0}}\left(1+\frac{1}{p^{2s}}+\frac{1}{p^{4s}}+\dots\right)\\
    &\prod_{p \in P_{1,ram}}\left(1+\frac{2}{p^{s}}+\frac{3}{p^{2s}}+\dots\right)\prod_{p \in P_{0,ram}}\left(1+\frac{1}{p^{s}}+\frac{1}{p^{2s}}+\dots\right).
\end{align*}
First, we may bound $Z_{2,x,\chi_0}(s)$ independently of $x$ from above by $(1-2^{-1})^2 = 4$. We may rewrite the rest of the product as $$\prod_{p \in P_{1,1}}(1-p^{-s})^{-4}(1-p^{-2s})\prod_{p \in P_{1,0} \cup P_{0,1}}(1-p^{-2s})^{-2}(1+p^{-2s})\prod_{p \in P_{0,0}}(1-p^{-2s})^{-1}$$\begin{equation}
    \label{eqn:temp5}
    \prod_{p \in P_{1,ram}}(1-p^{-s})^{-2}\prod_{p \in P_{0,ram}}(1-p^{-2s})^{-1}.
\end{equation} Then, one may check local factor by local factor that $$\zeta(s)L(\chi_{-1},s)L(\chi_x,s)L(\chi_{-1}\chi_x,s)\prod_{p \in P_{1,1}\cup P_{0,0}}(1-p^{-2s})\prod_{p \in P_{1,0}\cup P_{0,1}}(1+p^{-2s})$$$$\prod_{p \in P_{1,ram}}(1-p^{-s})^{-2}\prod_{p \in P_{0,ram}}(1-p^{-2s})^{-1}$$ is equal to the product in Equation \ref{eqn:temp5}. Now at $s=1$, there's a pole with residue which is bounded above by $$ O(L(\chi_x,1)L(\chi_{-1}\chi_x,1)\prod_{p|x}(1+p^{-1})^2).$$ By the class number formula, $$L(\chi_{-1}\chi_x,1)=\frac{\pi |\Cl(\bbQ[\sqrt{-x}])|}{\sqrt x}$$ for all $0<x \neq 3$, and for $x=3$, it's just $\frac{1}{3}$ of the above. So, the residue of $L(\chi_{\triv},s)$ at $s=1$ is bounded above by $$O(\frac{L(\chi_x,1)\prod_{p|x}(1+p^{-1})^2|\Cl(\bbQ[\sqrt{-x}])|}{\sqrt x}).$$ Thus, the residue of $L_x(s)$ is $$O(\frac{L(\chi_x,1)\prod_{p|x}(1+p^{-1})^2}{\sqrt x}).$$
\end{proof}
\begin{prop}
\label{prop:estimate_on_sum_of_F(n)Gx(n)}
Assuming GRH, for all $\epsilon>0$, $\sum_{n=1}^X F(n)G_x(n)=X\Res_{s=1} L_x(s)+O_{\epsilon}(X^{1/2+\epsilon}x^{\epsilon})$.
\end{prop} 
\begin{proof}
We use the Wiener-Ikehara Tauberian theorem. We need to bound $L_x(s)$ in the range $1/2+\epsilon < \Re(s) \leq 1$. Now as in the proof of Lemma \ref{lem:Lxs_meromorphic_continuation}, we have that $L_0(\chi,s)$ is equal to $L(\chi,s)L(\chi \chi_{-1},s)$ up to products of following forms:
\begin{enumerate}
    \item $\prod_p (1+\frac{a_{p,2}}{p^{2s}}+\frac{a_{p,3}}{p^{3s}}+\dots)$, where $f(i)=\max_p |a_{p,i}|$ satisfies $f(i) \leq P(i)$ for some polynomial $P$ independent of $x$.
    \item local factors of the primes ramified in either $\bbQ[i]$ or $\bbQ[\sqrt{-x}]$.
\end{enumerate} Now, for $\Re(s)=1/2+\epsilon$, $$\left|\prod_p (1+\frac{a_{p,2}}{p^{2s}}+\frac{a_{p,3}}{p^{3s}}+\dots)\right| \leq \prod_p (1+\frac{P(2)}{p^{1+2\epsilon}}+\frac{P(3)}{p^{3/2+3\epsilon}}+\dots) = O_{\epsilon}(1).$$ For the local factors of the ramified primes, in the case where $p \in P_{1,ram}$, for example, we have $$|(1+\frac{2 \chi(\pi)}{p^s}+\frac{3\chi(\pi^2)}{p^{2s}}+\dots)| = |(1-\chi(\pi)p^{-s})|^{-2} \leq (1-p^{-1/2})^{-2} \leq 12.$$ Thus, the contribution coming from the product over these local factors is bounded by $12^{\omega(x)+1} = O(x^{\epsilon})$. Thus, for $\Re(s)=1/2+\epsilon$, $|L_0(\chi,s)|=O_{\epsilon}(x^{\epsilon}|L(\chi,s)L(\chi \chi_{-1},s)|)$. Now $L(\chi,s)$ is an $L$-function with analytic conductor $O(x)$. So, by GRH, we have that for $s=1/2+it$, $|L(\chi,s)| = O_{\delta}((xt)^\delta)$ for any $\delta>0$. Thus, by the Phragm\'en-Lindel\"of Principle, we also have that $|L(\chi,s)| = O_{\delta,\epsilon}((xt)^\delta)$ for $s=1/2+\epsilon+it$. The same holds for $|L(\chi\chi_{-1},s)|$. Thus, we obtain the bound \begin{equation}
\label{eqn:temp6}
    |L_0(\chi,s)| = O_{\delta,\epsilon}(x^{\epsilon}(xt)^{\delta}) 
\end{equation} for $s=1/2+\epsilon+it$ and any $\epsilon>0$. Since $L_x(s)$ is an average of $L_0(\chi,s)$ as $\chi$ varies over $\widehat{\Cl(\bbQ[\sqrt{-x}])}$, we have the same bound for $|L_x(s)|$.

Now we are in the hypothesis of the Wiener-Ikehara Tauberian theorem as in Theorem \ref{thm:appendix_wiener_ikehara_with_error} with $\kappa=1$ and $r$ the implied constant in Equation \ref{eqn:temp6}. So, combining with Lemmas \ref{lem:Lxs_residue} and \ref{lem:Lxs_meromorphic_continuation}, we get that the sum of the coefficients of $L_x(s)$ up to $X$ is $$X\Res_{s=1}L_x(s)+O_{\epsilon,\delta}(X^{1/2+\epsilon}x^{\delta+\epsilon}).$$ We are done after choosing $\delta = \epsilon$.
\end{proof}

\begin{cor}
\label{cor:upper_bound_Sx}
Assuming GRH, for all $\epsilon>0$, $|S_x|=O_{\epsilon}(B\frac{L(\chi_x,1)\prod_{p|x}(1+p^{-1})^2}{x}+B^{1/2+\epsilon}x^{-1/4+\epsilon})$
\end{cor}
\begin{proof}
Combine Lemma \ref{lem:Lxs_residue} and Proposition \ref{prop:estimate_on_sum_of_F(n)Gx(n)}, and apply it to the case where $X=\frac{B}{\sqrt{x}}$.
\end{proof}
To conclude, we just need to sum over $x \leq B^{2/3-\epsilon}$.

\begin{proof}[Proof of Proposition \ref{prop:sum_of_Sx_for_x_small}]
By Corollary \ref{cor:upper_bound_Sx}, we just need to upper bound $$\sum_{x \geq 2, \text{ squarefree}}^{B^{2/3-\epsilon}}O_{\delta}\left(B\frac{L(\chi_x,1)\prod_{p|x}(1+p^{-1})^2}{x}+B^{1/2+\delta}x^{-1/4+\delta}\right)$$ for some small $\delta$ to be chosen later. Applying the sum, we obtain that the error term is $B^{1/2+\delta}B^{(2/3-\epsilon)(3/4+\delta)}=B^{1+\frac{5}{3}\delta-\frac{3}{4}\epsilon-\epsilon\delta}$. So, making the choice of $\delta=\frac{9}{20}\epsilon$, we get an error term of $O(B^{1-\epsilon\delta})=O(B)$, as desired. Now we may focus on our first term.

We wish to estimate $$\sum_{x \geq 2, \text{ squarefree}}^{B^{2/3-\epsilon}}B\frac{L(\chi_x,1)\prod_{p|x}(1+p^{-1})^2}{x}.$$ We have $\prod_{p|x}(1+p^{-1})^2 =O(\prod_{p|x}(1+2p^{-1}))=O(\sum_{d|x}\tau(d)/d)$. So, it suffices to bound $$\sum_{x \geq 2, \text{ squarefree}}^{B^{2/3-\epsilon}}\frac{L(\chi_x,1)\sum_{d|x}\tau(d)/d}{x}$$ by $O(\log B)$. We have \begin{align*}
    \sum_{x \geq 2, \text{ squarefree}}^{B^{2/3-\epsilon}}\frac{L(\chi_x,1)\sum_{d|x}\tau(d)/d}{x}&=\sum_{d=1}^{B^{2/3-\epsilon}}\frac{\tau(d)}{d}\sum_{x' \geq 2, dx' \text{ squarefree}}^{B^{2/3-\epsilon}d^{-1}}\frac{L(\chi_{dx'},1)}{dx'}\\
    &=\sum_{d=1}^{B^{2/3-\epsilon}}\frac{\tau(d)}{d^2}\sum_{x' \geq 2}^{B^{2/3-\epsilon}d^{-1}}\frac{|\mu(dx')|L(\chi_{dx'},1)}{x'}\\
    &=\sum_{d=1}^{B^{2/3-\epsilon}}\frac{\tau(d)|\mu(d)|}{d^2}\sum_{x' \geq 2, \gcd(d,x')=1}^{B^{2/3-\epsilon}d^{-1}}\frac{|\mu(x')|L(\chi_{dx'},1)}{x'}\\
    &=\sum_{d=1}^{B^{2/3-\epsilon}}\frac{\tau(d)|\mu(d)|}{d^2}\sum_{x' \geq 2}^{B^{2/3-\epsilon}d^{-1}}\frac{|\mu(x')|L(\chi_{dx'},1)}{x'},
\end{align*} where the last inequality used the fact that $L(\chi_{dx'},1) \geq 0$ for all $d,x'$.


We will record the bound of the inner sum in a lemma, to be used later as well.

\begin{lemma}
\label{lem:ayoub_sum}
$$\sum_{x' \geq 2}^{X}\frac{|\mu(x')|L(\chi_{dx'},1)}{x'}=O(\log X + d^{1/2}\log(d))$$
\end{lemma}

\begin{proof}
We may use summation by parts to get $$\sum_{x' \geq 2}^{X}\frac{|\mu(x')|L(\chi_{dx'},1)}{x'}=O\left(\sum_{k=1}^{X} \frac{\sum_{x'=1}^k |\mu(x')|L(\chi_{dx'},1)}{k^2}\right).$$ 


For the inner sum, we have $$\sum_{x'=1}^k |\mu(x')|L(\chi_{dx'},1)=\sum_{x'=1}^k |\mu(x')|\sum_{n=1}^\infty \frac{\left(\frac{dx'}{n}\right)}{n}.$$ Now, we will split the infinite sum over $n$ into two parts. The first part will be \begin{equation}
\label{eqn:temp4}
    \sum_{x'=1}^k |\mu(x')|\sum_{n=k}^{\infty} \frac{\left(\frac{dx'}{n}\right)}{n}=\sum_{x'=1}^k |\mu(x')|O(\frac{\sqrt{dx'} \log (dx')}{k})=O(d^{1/2}k^{1/2}\log(dk))
\end{equation} where the first equality is obtained through summation by parts and Polya-Vinogradov. So, we are left with the other part, $$\sum_{x'=1}^k |\mu(x')|\sum_{n=1}^k \frac{(\frac{dx'}{n})}{n}.$$ Here, we exchange summation and consider \begin{equation} \label{eqn:temp5151}
    \sum_{n=1}^k \frac{1}{n}\sum_{x'=1}^k |\mu(x')|\left(\frac{dx'}{n}\right)=\sum_{n=1}^k \frac{1}{n} \left(\frac{d}{n}\right)\sum_{x'=1}^k |\mu(x')|\left(\frac{x'}{n}\right).
\end{equation} We separate this sum into two cases, depending on whether $n$ is a square. So, we separate Equation \ref{eqn:temp5151} into $$\sum_{n=1, n \text{ is a square}}^k \frac{1}{n} \left(\frac{d}{n}\right)\sum_{x'=1}^k |\mu(x')|\left(\frac{x'}{n}\right)+\sum_{n=1, n \text{ is not a square}}^k \frac{1}{n} \left(\frac{d}{n}\right)\sum_{x'=1}^k |\mu(x')|\left(\frac{x'}{n}\right)$$ First we deal with the case where $n$ is a square, where we are happy to use the trivial estimate $|\left(\frac{d}{n}\right)| = |\left(\frac{x'}{n}\right)| = 1$, and we obtain $$\sum_{n=1, n \text{ is a square}}^k \frac{1}{n} \left(\frac{d}{n}\right)\sum_{x'=1}^k |\mu(x')|\left(\frac{x'}{n}\right) \leq \sum_{n=1, n \text{ is a square}}^k \frac{k}{n} = O(\sum_{n'=1}^{\sqrt{k}} \frac{k}{n'^2}) = O(k).$$

Now we turn to the case where $n$ is not a square and estimate $$\sum_{n=1, n \text{ is not a square}}^k \frac{1}{n} \left(\frac{d}{n}\right)\sum_{x'=1}^k |\mu(x')|\left(\frac{x'}{n}\right).$$Using the fact that $|\mu(x')| = \sum_{l^2| x'} \mu(l)$, we obtain $$\sum_{x'=1}^k |\mu(x')|\left(\frac{x'}{n}\right) = \sum_{x'=1}^k \sum_{l^2|x'} \mu(l) \left(\frac{x'}{n}\right) = \sum_{l=1}^{\sqrt{k}} \mu(l) \sum_{x''=1}^{k/l^2} \left( \frac{x''}{n}\right),$$ where in the last equality, we made the substitution $x'' = l^2 x'$. By Polya-Vinogradov, if $n$ is not a square, $\left\lvert\sum_{x'=1}^{k/l^2}\left(\frac{x'}{n}\right)\right\lvert=O(\sqrt{n} \log n)$, and a trivial estimate gives us that it is $O(k/l^2)$. Hence, $$\left\lvert\sum_{x'=1}^{k/l^2}\left(\frac{x'}{n}\right)\right\lvert=O(\min(\sqrt{n}\log n, k/l^2)).$$ Thus, \begin{align*}
    \sum_{l=1}^{\sqrt{k}} \mu(l) \sum_{x''=1}^{k/l^2} \left( \frac{x''}{n}\right) &= O( \sum_{l=1}^{\sqrt{k}} \min(\sqrt{n}\log n, k/l^2)) \\
    &= O\left(\sum_{l=1}^{\sqrt{k}/(\sqrt{n} \log n))^{1/2}} \sqrt{n} \log n + \sum_{l=\sqrt{k}/(\sqrt{n} \log n))^{1/2}}^{\sqrt{k}} k/l^2 \right)\\
    &= O(k^{1/2}n^{1/4}\log^{1/2}(n))
\end{align*} when $n$ is not a square. Using this estimate, we have that \begin{align*}
    \sum_{n=1, n \text{ is not a square}}^k \frac{1}{n} \left(\frac{d}{n}\right)\sum_{x'=1}^k |\mu(x')|\left(\frac{x'}{n}\right) &= O\left(\sum_{n=1, n \text{ is not a square}}^k \frac{k^{1/2}n^{1/4}\log^{1/2}(n)}{n}\right) \\
    &= O(k^{3/4} \log^{1/2}(k)) \\
    &= O(k).
\end{align*}


Adding to Equation \ref{eqn:temp4}, we obtain $$\sum_{x'=1}^k |\mu(x')|L(\chi_{dx'},1)=O(k+(dk)^{1/2}\log(dk)).$$


Finally, we sum over $k$. We have \begin{align*}
    \sum_{k=1}^{X} \frac{\sum_{x'=1}^k |\mu(x')|L(\chi_{dx'},1)}{k^2}&=\sum_{k=1}^{X}\frac{O(k+d^{1/2}k^{1/2}\log (dk))}{k^2}\\
    &=O(\log X+d^{1/2}\log(d)).
\end{align*}
\end{proof}

With the lemma, we have $$\sum_{x' \geq 2}^{B^{2/3-\epsilon}d^{-1}}\frac{|\mu(x')|L(\chi_{dx'},1)}{x'}=O(\log (B^{2/3-\epsilon}d^{-1}))+O(d^{1/2}\log(d)).$$ Thus, \begin{align*}
    \sum_{d=1}^{B^{2/3-\epsilon}}\frac{\tau(d)|\mu(d)|}{d^2}\sum_{x'=1,(x',d)=1}^{B^{2/3-\epsilon}d^{-1}}\frac{|\mu(x')|L(\chi_{dx'},1)}{x'}&=\sum_{d=1}^{B^{2/3-\epsilon}}\frac{\tau(d)|\mu(d)|}{d^2}(O(\log (B^{2/3-\epsilon}d^{-1}))+O(d^{1/2}\log(d)))\\
    &=O(\log(B^{2/3-\epsilon})+O(1)\\
    &=O(\log(B)).
\end{align*}
\end{proof}

\section{Estimate of $|S_x|$ for $-1 \geq x \geq B^{2/5-\epsilon}$}\label{sec:S_x bound negative}

In this section, we prove Proposition \ref{prop:upper_bound_small_x_negative}. Recall the definition $$S_x = \{(a',b,c'): \sqf(a'^2+b^2-c'^2)=x,  \max(|a'|,|b|,|c'|)^2 \sqrt{|x|} \leq B\}.$$ We will introduce a new set $S_x'$ which is easier to count and satisfies $|S_x| \leq |S_x'|$. We first write $a'^2+b^2-c'^2=xy^2$, for some $y$. Rearranging the terms, we get $a'^2-xy^2=c'^2-b^2=(c'-b)(c'+b)$. Making the linear change of variables $b' = c'-b$ and $c'' = c'+b$, we have that $S_x$ is in bijection with $\{(a',b',c''): \sqf(a'^2-b'c'')=x, \max(|a'|,|\frac{c''-b'}{2}|,|\frac{c''+b'}{2}|)^2 \sqrt{|x|} \leq B, c'' \equiv b' \mod 2\}$. For the sake of an upper bound, we may forget the condition $c'' \equiv b' \mod 2$, and also we have that $$\max(|a'|,|b'|,|c''|) = \max(|a'|,|c'-b|,|c'+b|) \leq 2 \max(|a'|,|b|,|c'|) = 2 \max(|a'|,|\frac{c''-b'}{2}|,|\frac{c''+b'}{2}|).$$ Hence, defining $$S_x' := \{(a',b',c''): \sqf(a'^2-b'c'')=x, \max(|a'|,|b'|,|c''|)^2 \sqrt{|x|} \leq B\},$$ we have that $|S_x| \leq |S_x'|$, and so now it suffices to bound $S_x'$.

Note that by definition of $S_x'$, any $(a',b',c'') \in S_x'$ must satisfy that $b'c'' \leq B/\sqrt{|x|}$. Now, writing $a'^2-b'c''=xy^2$, and rearranging, we get $b'c''=a'^2-xy^2$. Now define $\tau'(n)$ to be the number of divisors $d$ of $n$ so that both $d$ and $n/d$ are bounded by $(B/\sqrt{|x|})^{1/2}$. Then, $\tau'(n)G_{-x}(n)$ is the number of solutions $a',b',c'',y$ to the equation $b'c''=a'^2-xy^2=n$. So, for a fixed negative $x$, we have $|S_x'|=\sum_{n=1}^{B/\sqrt{|x|}}\tau'(n)G_{-x}(n)$. However, $\tau'$ is not a multiplicative function like $F(n)$, so we will approach this summation slightly differently. Denote by $\mathbbm{1}_{[1,(B/\sqrt{|x|})^{1/2}]}$ the characteristic function of the interval $[1,(B/\sqrt{|x|})^{1/2}]$. Then, we have \begin{align*}
    \sum_{n=1}^{B/\sqrt{|x|}}\tau'(n)G_{-x}(n)&=\sum_{n=1}^{B/\sqrt{|x|}}G_{-x}(n)\sum_{d|n} \mathbbm{1}_{[1,(B/\sqrt{|x|})^{1/2}]}(d)\mathbbm{1}_{[1,(B/\sqrt{|x|})^{1/2}]}(n/d)\\
    &=\sum_{d=1}^{(B/\sqrt{|x|})^{1/2}} \sum_{n'=1}^{(B/\sqrt{|x|})/d}G_{-x}(dn')\mathbbm{1}_{[1,(B/\sqrt{|x|})^{1/2}]}(n')\\
    &=\sum_{d=1}^{(B/\sqrt{|x|})^{1/2}} \sum_{n'=1}^{(B/\sqrt{|x|})^{1/2}}G_{-x}(dn').
\end{align*} Now we will use the theory of $L$-functions and Tauberian theorems to estimate the inner sum, and subsequently the outer sum. 

Let $x'=-x$. We need to estimate $$\sum_{d=1}^{(B/\sqrt{x'})^{1/2}} \sum_{n'=1}^{(B/\sqrt{x'})^{1/2}}G_{x'}(dn')$$ as $x'$ ranges between the squarefree numbers in the interval $[2,B^{2/5-\epsilon}]$. Fixing $d$, we need to find an estimate on $$\sum_{n'=1}^{(B/\sqrt{x'})^{1/2}}G_{x'}(dn').$$ We do so by considering the $L$-function $L_{x',d}(s)=\sum_{n=1}^\infty \frac{a_{x',d,n}}{n^s}$, where $$a_{x',d,n} = \begin{cases}|\{\text{ideals } I \subset \cO_{\bbQ[\sqrt{-x'}]}: N(I)=n, I \text{ is principal}\}| & d|n \\
0 & \text{ otherwise}
\end{cases},$$ so that $$\sum_{n'=1}^{(B/\sqrt{x'})^{1/2}}G_{x'}(dn')=\sum_{n=1}^{d(B/\sqrt{x'})^{1/2}} a_{x',d,n}.$$ We will apply the Wiener-Ikehara Tauberian theorem here. Thus, as we did in the previous section, we will do the following. \begin{enumerate}
    \item We show that $L_{x',d}(s)$ can be meromorphically continued to $\Re(s)>1/2$, with a simple pole at $s=1$.
    \item We find the residue of $L_{x',d}(s)$ at $s=1$.
    \item We use the Wiener-Ikehara Tauberian theorem as in Theorem \ref{thm:appendix_wiener_ikehara_with_error} to obtain an estimate on the sum of the coefficients on $L_{x',d}(s)$, with an acceptable error term.
\end{enumerate}

We will follow similar steps as the last section. Let $\chi \in \widehat{\Cl(\bbQ[\sqrt{-x'}])}$. Then, define $L_d(\chi,s)=\sum_{n=1}^{\infty} \frac{a_{\chi,d,n}}{n^s}$, where $$a_{\chi,d,n} = \begin{cases} \sum_{\text{ideals } I \subset \cO_{\bbQ[\sqrt{-x'}]}: N(I)=n} \chi(I) & d|n \\
0 & \text{ otherwise}
\end{cases},$$ so that $L_{x',d}(s) = \frac{1}{|\Cl(\bbQ[\sqrt{-x'}])|}\sum_{\chi \in \widehat{\Cl(\bbQ[\sqrt{-x'}])|}}L_d(\chi,s).$


Let \begin{align*}
    &P_{1}=\{\text{primes }p | p \text{ splits in }\bbQ[\sqrt{-x'}]\}\\
    &P_{0}=\{\text{primes }p | p \text{ is inert in }\bbQ[\sqrt{-x'}]\}\\
    &P_{ram}=\{\text{primes }p | p \text{ ramifies in }\bbQ[\sqrt{-x'}]\}.
\end{align*} Also let $v_p(d)$ denote the exponent of $p$ in $d$.

Define $$L_{x',d}(s)=\frac{1}{|\Cl(\bbQ[\sqrt{-x'}])|}\sum_{\chi \in \widehat{\Cl(\bbQ[\sqrt{-x'}])}} L_d(\chi,s)$$ where $$L_d(\chi,s)=\prod_{p \in P_1}\sum_{i=v_p(d)}^{\infty} \frac{ \sum_{j=0}^{i}\chi(\pi^j \bar \pi ^{i-j})}{p^{is}}\prod_{p \in P_0}\sum_{i=\lceil v_p(d)/2 \rceil}^{\infty} \frac{\chi(p^{2i})}{p^{2is}}\prod_{p \in P_{ram}}\sum_{i=v_p(d)}^{\infty} \frac{\chi(\pi^i)}{p^{is}},$$ where $\pi, \bar \pi$ are primes lying above $p$.

There are two remarks on this function. One remark is that $L_1(\chi,s)$ is just the Hecke L-function $L(\chi,s)$. The second remark on this function is that there is a relationship between $L_1(\chi,s)$ and $L_d(\chi,s)$ given by \begin{align*}
    L_d(\chi,s) &= L_1(\chi,s) \prod_{p \in P_1}\frac{\sum_{i=v_p(d)}^{\infty} \frac{ \sum_{j=0}^{i}\chi(\pi^j \bar \pi ^{i-j})}{p^{is}}}{\sum_{i=0}^{\infty} \frac{ \sum_{j=0}^{i}\chi(\pi^j \bar \pi ^{i-j})}{p^{is}}}\prod_{p \in P_0}\frac{\sum_{i=\lceil v_p(d)/2 \rceil}^{\infty} \frac{\chi(p^{2i})}{p^{2is}}}{\sum_{i=0}^{\infty} \frac{\chi(p^{2i})}{p^{2is}}}\prod_{p \in P_{ram}}\frac{\sum_{i=v_p(d)}^{\infty} \frac{\chi(\pi^i)}{p^{is}}}{\sum_{i=0}^{\infty} \frac{\chi(\pi^i)}{p^{is}}}\\
    &=L_1(\chi,s)Q_{d,1}(\chi,s)Q_{d,0}(\chi,s)Q_{d,ram}(\chi,s),
\end{align*} where $Q_{d,1}$, $Q_{d,0}$, and $Q_{d,ram}$ are holomorphic over $Re(s)>0$.

\begin{lemma}
$L_{x',d}(s)$ is meromorphic over $Re(s)>0$, and only has a pole at $s=1$ of residue $\frac{1}{|\Cl(\bbQ[\sqrt{-x'}])|}\Res_{s=1}L_d(\chi_{\triv},s)$.
\label{lem:Lxds_meromorphic}
\end{lemma}
\begin{proof}
$L_{x',d}(s)$ is an average of $L_{d}(\chi,s)$ as $\chi$ varies over $\widehat{\Cl(\bbQ[\sqrt{-x}])}$. Now $L_{d}(\chi,s)=L_1(\chi,s)Q_{d,1}(\chi,s)Q_{d,0}(\chi,s)Q_{d,ram}(\chi,s)$, with each factor holomorphic over $Re(s)>0$, except possibly $L_1(\chi,s)$ when $\chi=\chi_{\triv}$, in which case there is a pole at $s=1$. Thus, the residue of $L_{x',d}(s)$ at $s=1$ is $\frac{1}{|\Cl(\bbQ[\sqrt{-x'}])|}\Res_{s=1}L_d(\chi_{\triv},s)$.
\end{proof}

\begin{lemma}
\label{lem:Lxds_residue}
$$\Res_{s=1}L_{x',d}(s) = O\left( \frac{1}{\sqrt{x'}} \prod_{p \in P_1} \frac{v_p(d)+1}{p^{v_p(d)}} \prod_{p \in P_0}\frac{1}{p^{2 \lceil v_p(d)/2 \rceil}} \prod_{p \in P_{ram}} \frac{1}{p^{v_p(d)}}\right).$$
\end{lemma}
\begin{proof}
Since we're looking at the residue at $1$, and by Lemma \ref{lem:Lxds_meromorphic}, we just need to look at $L_d(\chi_{\triv},s)$. We have $$L_d(\chi_{\triv},s) =L_1(\chi_{\triv},s)Q_{d,1}(\chi_{\triv},s)Q_{d,0}(\chi_{\triv},s)Q_{d,ram}(\chi_{\triv},s).$$ Since all factors besides $L_1(\chi_{\triv},s)$ are entire, we focus on the residue of $L_1(\chi_{\triv},s)$ at $1$. However, this is just the Dedekind zeta function of $\bbQ[\sqrt{-x'}]$. Thus, by the class number formula, $L_1(\chi_{\triv},s)$ has residue at $1$ bounded by $C_0 \frac{|\Cl(\bbQ[\sqrt{-x'}])|}{\sqrt{x'}}$, where $C_0$ is a constant independent of $x'$. Thus, $L_d(\chi_{\triv},1)$ has residue bounded by $$C_0 \frac{|\Cl(\bbQ[\sqrt{-x'}])|}{\sqrt{x'}}Q_{d,1}(\chi_{\triv},1)Q_{d,0}(\chi_{\triv},1)Q_{d,ram}(\chi_{\triv},1).$$ Thus, $L_{d,x'}(s)$ has residue bounded by $$C_0 \frac{Q_{d,1}(\chi_{\triv},1)Q_{d,0}(\chi_{\triv},1)Q_{d,ram}(\chi_{\triv},1)}{\sqrt{x'}}.$$ 

Now, we investigate $Q_{d,1}(\chi_{\triv},1)$. A local factor of $Q_{d,1}(\chi_{\triv},1)$ is of the form $\frac{\sum_{i=v_p(d)}^{\infty} \frac{ i+1}{p^{i}}}{\sum_{i=0}^{\infty} \frac{i+1}{p^{i}}}$, and we have the bound $$\frac{\sum_{i=v_p(d)}^{\infty} \frac{ i+1}{p^{i}}}{\sum_{i=0}^{\infty} \frac{i+1}{p^{i}}}=p^{-v_p(d)}\frac{\sum_{i=0}^{\infty} \frac{v_p(d)+i+1}{p^{i}}}{\sum_{i=0}^{\infty} \frac{i+1}{p^{i}}}=(v_p(d)+1)p^{-v_p(d)}\frac{\sum_{i=0}^{\infty} \frac{1+i/(v_p(d)+1)}{p^{i}}}{\sum_{i=0}^{\infty} \frac{i+1}{p^{i}}}\leq (v_p(d)+1)p^{-v_p(d)}.$$ Similar analyses of $Q_{d,0}(\chi_{\triv},1)$ and $Q_{d,ram}(\chi_{\triv},1)$ yields that $$Q_{d,1}(\chi_{\triv},1)Q_{d,0}(\chi_{\triv},1)Q_{d,ram}(\chi_{\triv},1) \leq \prod_{p \in P_1} \frac{v_p(d)+1}{p^{v_p(d)}} \prod_{p \in P_0}\frac{1}{p^{2 \lceil v_p(d)/2 \rceil}} \prod_{p \in P_{ram}} \frac{1}{p^{v_p(d)}}.$$
\end{proof}

\begin{lemma}
Assuming GRH, for all $\epsilon_1>0$, $$L_{x',d}(\chi,1/2+it)=O_{\epsilon_1}(|x't|^{\epsilon_1}d^{-1/2}\prod_{p|d} \frac{v_p(d)+1}{(1-p^{-1/2})^{4}}).$$
\label{lem:Lxds_subconvexity}
\end{lemma}
\begin{proof}
Since $L_{x',d}$ is an average of $L_d(\chi,s)$, we just need to show the same bound for $L_d(\chi,s)$. $$L_{d}(\chi,s)=L_{1}(\chi,s)Q_{d,1}(\chi,s)Q_{d,0}(\chi,s)Q_{d,ram}(\chi,s).$$ We will bound $L_{1}(\chi,1/2+it)$ and $Q_{d,1}(\chi,1/2+it)Q_{d,0}(\chi,1/2+it)Q_{d,ram}(\chi,1/2+it)$ separately.

The analytic conductor of $L_1(\chi,1/2+it)$ is $x'$. Since we are assuming GRH, we may use \cite[Cor. 5.20]{iwaniec_kowalski_2004} to obtain $L_1(\chi,1/2+it)=O_{\epsilon_1}(|x't|^{\epsilon_1})$.

For $Q_{d,1}(\chi,1/2+it)Q_{d,0}(\chi,1/2+it)Q_{d,ram}(\chi,1/2+it)$, we will bound as follows. Consider one local factor from $Q_{d,1}$. For example, $$\frac{\sum_{j=v_p(d)}^{\infty} \frac{\sum_{l=0}^j\chi(\pi^l \bar \pi^{j-l})}{p^{(1/2+it)j}}}{\sum_{j=0}^{\infty} \frac{\sum_{l=0}^j\chi(\pi^l \bar \pi^{j-l})}{p^{(1/2+it)j}}}.$$ We may upper bound the numerator by $$\sum_{j=v_p(d)}^{\infty} \frac{j+1}{p^{j/2}}=\frac{1}{p^{v_p(d)/2}}\sum_{j=0}^{\infty} \frac{j+1+v_p(d)}{p^{j/2}}=\frac{1}{p^{v_p(d)/2}}(\frac{1}{(1-p^{-1/2})^2}+\frac{v_p(d)}{1-p^{-1/2}}).$$ As for the denominator, we may write it as $$(1-\chi(\pi)p^{-1/2+it})^{-1}(1-\overline{\chi(\pi)}p^{-1/2+it})^{-1},$$ with the individual products lower bounded by $(1+p^{-1/2})^{-1}$, and so we obtain a upper bound of $$\frac{1}{p^{v_p(d)/2}}(1+p^{-1/2})^2((1-p^{-1/2})^{-2}+v_p(d)(1-p^{-1/2})^{-1}) \leq \frac{v_p(d)+1}{p^{v_p(d)/2}(1-p^{-1/2})^4}$$ for the local factor at $p$. So, $Q_{d,1}(\chi,1/2+it)$ may be bounded by $$\prod_{p|d, p \in P_1}\frac{v_p(d)+1}{p^{v_p(d)/2}(1-p^{-1/2})^4}.$$ We may apply a similar argument to $Q_{d,0}$ and upper bound $Q_{d,0}(\chi,1/2)$ by $$\prod_{p|d, p \in P_0}\frac{1}{p^{2\lceil v_p(d)/2 \rceil}} \frac{(1-p^{-1})^{-1}}{(1+p^{-1})^{-1}} \leq \prod_{p|d, p \in P_0}\frac{1}{p^{\lceil v_p(d)/2 \rceil}} (1-p^{-1})^{-2} \leq \prod_{p|d, p \in P_0}\frac{v_p(d)+1}{p^{v_p(d)/2}(1-p^{-1/2})^4}.$$ As for $Q_{d,ram}$, we obtain an upper bound of $$\prod_{p|d, p \in P_{ram}}\frac{1}{p^{v_p(d)/2}} (1-p^{-1/2})^{-2} \leq \prod_{p|d, p \in P_{ram}}\frac{v_p(d)+1}{p^{v_p(d)/2}(1-p^{-1/2})^4}.$$  Thus, we obtain that $$|Q_{d,1}(\chi,1/2+it)Q_{d,0}(\chi,1/2+it)Q_{d,ram}(\chi,1/2+it)| \leq d^{-1/2}\prod_{p|d} \frac{v_p(d)+1}{(1-p^{-1/2})^{4}}$$

\end{proof}

We now have all the ingredients for the Wiener-Ikehara Tauberian Theorem.

\begin{prop}
Assuming GRH, for all $\epsilon_1>0$, $$\sum_{d=1}^X\sum_{n=1}^X G_{x'}(dn)=O_{\epsilon_1}\left(\frac{|\Cl(\bbQ[\sqrt{-x'}])|}{x'\prod_{p|x'}(1-1/p)}X^2+x'^{\epsilon_1}X^{3/2+\epsilon_1}\right).$$
\label{prop:temp1}
\end{prop}
\begin{proof}
Using the conclusions of Lemmas \ref{lem:Lxds_meromorphic}, \ref{lem:Lxds_residue}, and \ref{lem:Lxds_subconvexity}, we may apply Theorem \ref{thm:appendix_wiener_ikehara_with_error} in the appendix, with $b=1$, $a=1$, $\kappa=1$, and $r=d^{-1/2}\prod_{p|d} \frac{v_p(d)+1}{(1-p^{-1/2})^{4}}$. We then obtain $$\sum_{n=1}^X G_{x'}(dn)=O_{\epsilon_1}(\frac{Xd}{\sqrt{x'}}\prod_{p \in P_1} \frac{v_p(d)+1}{p^{v_p(d)}} \prod_{p \in P_0}\frac{1}{p^{2 \lceil v_p(d)/2 \rceil}} \prod_{p \in P_{ram}} \frac{1}{p^{v_p(d)}}+x'^{\epsilon_1}(Xd)^{1/2+\epsilon_1}d^{-1/2}\prod_{p|d} \frac{v_p(d)+1}{(1-p^{-1/2})^{4}}).$$ 

First we deal with the first term.

$$\sum_{d=1}^X \frac{Xd}{\sqrt{x'}}\prod_{p \in P_1} \frac{v_p(d)+1}{p^{v_p(d)}} \prod_{p \in P_0}\frac{1}{p^{2 \lceil v_p(d)/2 \rceil}}\prod_{p \in P_{ram}} \frac{1}{p^{v_p(d)}}= \frac{X}{\sqrt{x'}}\sum_{d=1}^X\prod_{p \in P_1} (v_p(d)+1) \prod_{p \in P_0}\frac{1}{p^{2 \lceil v_p(d)/2 \rceil-v_p(d)}}.$$

To estimate the sum, we may again use the Tauberian theorem (Theorem \ref{thm:appendix_wiener_ikehara_with_error}) associated to the $L$-function given by the Euler product $$\prod_{p \in P_1}(1+\frac{2}{p^s}+\frac{3}{p^{2s}}+\dots) \prod_{p \in P_0} (1+\frac{2}{p^{2s}}+\frac{2}{p^{4s}}+\dots).$$ But this is almost exactly the Dedekind Zeta Function of $\bbQ[\sqrt{-x'}]$. Indeed, the Dedekind Zeta Function of $\bbQ[\sqrt{-x'}]$ is $$\prod_{p \in P_1}(1+\frac{2}{p^s}+\frac{3}{p^{2s}}+\dots) \prod_{p \in P_0} (1+\frac{1}{p^{2s}}+\frac{1}{p^{4s}}+\dots)\prod_{p \in P_{ram}}\left(1+\frac{1}{p^s}+\frac{1}{p^{2s}}+\dots\right).$$ So, by the class number formula and the Tauberian theorem, $$\sum_{d=1}^X\prod_{p \in P_1} (v_p(d)+1)\prod_{p \in P_0}\frac{1}{p^{2 \lceil v_p(d)/2 \rceil-v_p(d)}} = O_{\epsilon_1}\left(\frac{|\Cl(\bbQ[\sqrt{-x'}])|}{\sqrt{x'}\prod_{p|x'}(1-1/p)}X+x'^{\epsilon_1}X^{1/2+\epsilon_1}\right).$$ Thus, our final estimate for the first term is $$O_{\epsilon_1}\left(\frac{|\Cl(\bbQ[\sqrt{-x'}])|}{x'\prod_{p|x'}(1-1/p)}X^2+X^{3/2+\epsilon_1}x'^{\epsilon_1}\right)$$

Now we deal with the second term, \begin{align*}
    \sum_{d=1}^Xx'^{\epsilon_1}(Xd)^{1/2+\epsilon_1}d^{-1/2}\prod_{p|d} \frac{v_p(d)+1}{(1-p^{-1/2})^{4}} &= x'^{\epsilon_1}X^{1/2+\epsilon_1}\sum_{d=1}^X d^{\epsilon_1}\prod_{p|d} \frac{v_p(d)+1}{(1-p^{-1/2})^{4}}\\ &\leq x'^{\epsilon_1}X^{1/2+2\epsilon_1} \sum_{d=1}^X\prod_{p|d} \frac{v_p(d)+1}{(1-p^{-1/2})^{4}}.
\end{align*} For the summation, we may again use the Tauberian theorem on the $L$-function associated to the Euler product $$\prod_{p}\left(1+\frac{2(1-p^{-1/2})^{-4}}{p^s}+\frac{3(1-p^{-1/2})^{-4}}{p^{2s}}+\dots\right)$$ Given that $(1-2^{-1/2})^{-4} \leq 136$, the coefficients of this $L$-function is certainly bounded by the coefficients of the $L$-function associated to the product $$\prod_{p}\left(1+\frac{2\cdot 136}{p^s}+\frac{3\cdot 136}{p^{2s}}+\dots\right).$$ The Tauberian theorem (Theorem \ref{thm:appendix_wiener_ikehara_with_error}) with $b=2 \cdot 136=272$ then gives us something on the order of $X (\log X)^{272}=O(X^{1+\epsilon_1})$. So, the final estimate of the second term is $O_{\epsilon_1}(x'^{\epsilon_1}X^{3/2+3\epsilon_1})$.

Consequently, $$\sum_{d=1}^X\sum_{n=1}^XG_{x'}(dn)=O_{\epsilon_1}\left(\frac{|\Cl(\bbQ[\sqrt{-x'}])|}{x'\prod_{p|x'}(1-1/p)}X^2+x'^{\epsilon_1}X^{3/2+\epsilon_1}\right).$$
\end{proof}

\begin{cor} Assuming GRH, for all $\epsilon_1>0$,
$$|S_x|=O_{\epsilon_1}\left(\frac{|\Cl(\bbQ[\sqrt{-x'}])|}{x'^{3/2}\prod_{p|x'}(1-1/p)}B+B^{3/4+\epsilon_1}x'^{-3/8+\epsilon_1}\right)$$
\end{cor}
\begin{proof}
Apply Proposition \ref{prop:temp1} to $X=(B/\sqrt{x'})^{1/2}$.
\end{proof}

Finally, we get our proposition
\begin{proof}[Proof of Proposition \ref{prop:upper_bound_small_x_negative}]
By the preceding corollary, we have $$\sum_{-B^{2/5-\epsilon}\leq x \leq -1, \text{squarefree}}|S_x|=O_{\epsilon_1}\left(\sum_{x'=1}^{B^{2/5-\epsilon}}\frac{|\Cl(\bbQ[\sqrt{-x'}])|}{x'^{3/2}\prod_{p|x'}(1-1/p)} B+\frac{B^{3/4+\epsilon_1}}{x'^{3/8+\epsilon_1}}\right).$$ Choosing $\epsilon_1 = \frac{25}{56}\epsilon$, we get $$\sum_{-B^{2/5-\epsilon}\leq x \leq -1, \text{squarefree}}|S_x|=O_{\epsilon}\left(B\sum_{x'=1, \text{ squarefree}}^{B^{2/5-\epsilon}}\frac{|\Cl(\bbQ[\sqrt{-x'}])|}{x'^{3/2}\prod_{p|x'}(1-1/p)}\right)+O_{\epsilon}(B).$$ Thus, we see that it suffices to show that $$\sum_{x'=1, \text{ squarefree}}^{B^{2/5-\epsilon}}\frac{|\Cl(\bbQ[\sqrt{-x'}])|}{x'^{3/2}\prod_{p|x'}(1-1/p)} = O(\log B).$$

Now it follows from the Class Number Formula and the fact that the Dedekind Zeta Function of a quadratic extension splits as a product of $L$-functions that $\frac{|\Cl(\bbQ[\sqrt{-x'}])|}{\sqrt{x'}} = O(L(\chi_{-x'},1))$, where we recall $\chi_{-x'}$ is defined to be the Jacobi Symbol $(\frac{-x'}{n})$. Thus, we may write $$\sum_{x'=1, \text{ squarefree}}^{B^{2/5-\epsilon}}\frac{|\Cl(\bbQ[\sqrt{-x'}])|}{x'^{3/2}\prod_{p|x'}(1-1/p)}=O\left(\sum_{x'=1}^{B^{2/5-\epsilon}}\frac{L(\chi_{-x'},1)|\mu(x')|}{x'\prod_{p|x'}(1-1/p)}\right).$$ Now, for $x'$ squarefree, $$\frac{1}{\prod_{p|x'}(1-1/p)}=\prod_{p|x'} (1+\frac{1}{p-1}) = \sum_{d|x'} \frac{1}{\phi(d)}.$$ So, \begin{align*}
    \sum_{x'=1}^{B^{2/5-\epsilon}}\frac{L(\chi_{-x'},1)|\mu(x')|}{x'\prod_{p|x'}(1-1/p)}&=\sum_{x'=1}^{B^{2/5-\epsilon}}\sum_{d|x'}\frac{L(\chi_{-x'},1)|\mu(x')|}{x'\phi(d)}\\
    &=O\left(\sum_{d=1}^{B^{2/5-\epsilon}} \sum_{x''=1}^{B^{2/5-\epsilon}/d}\frac{L(\chi_{-dx''},1)|\mu(dx'')|}{dx''\phi(d)}\right)\\
    &=O\left(\sum_{d=1}^{B^{2/5-\epsilon}} \frac{|\mu(d)|}{d\phi(d)}\sum_{x'=1, \gcd(d,x')=1}^{B^{2/5-\epsilon}/d}\frac{L(\chi_{-dx'},1)|\mu(x')|}{x'}\right).
\end{align*} Since $L(\chi_{-dx'},1)$ is always nonnegative, we have also that $$\sum_{x'=1, (d,x)=1}^{B^{2/5-\epsilon}/d}\frac{L(\chi_{-dx'},1)|\mu(x')|}{x'} \leq \sum_{x'=1}^{B^{2/5-\epsilon}/d}\frac{L(\chi_{-dx'},1)|\mu(x')|}{x'}.$$ By Lemma \ref{lem:ayoub_sum}, we have that $$\sum_{x'=1}^{B^{2/5-\epsilon}/d}\frac{L(\chi_{-dx'},1)|\mu(x')|}{x'} = O(\log(B^{2/5-\epsilon}))+O(d^{1/2}\log d).$$ So, $$\sum_{d=1}^{B^{2/5-\epsilon}} \frac{1}{d\phi(d)}\sum_{x'=1, \gcd(d,x')=1}^{B^{2/5-\epsilon}/d}\frac{L(\chi_{-dx'},1)|\mu(x')|}{x'} = \sum_{d=1}^{B^{2/5-\epsilon}} \frac{O(\log B+d^{1/2} \log d)}{d\phi(d)} = O(\log B)$$
\end{proof}

\section{Upper Bound of $|S_y|$ for $1 \leq y \leq B^{5/19-\epsilon}$} \label{sec:S_y bound}
\begin{figure}
    \centering
    \includegraphics[]{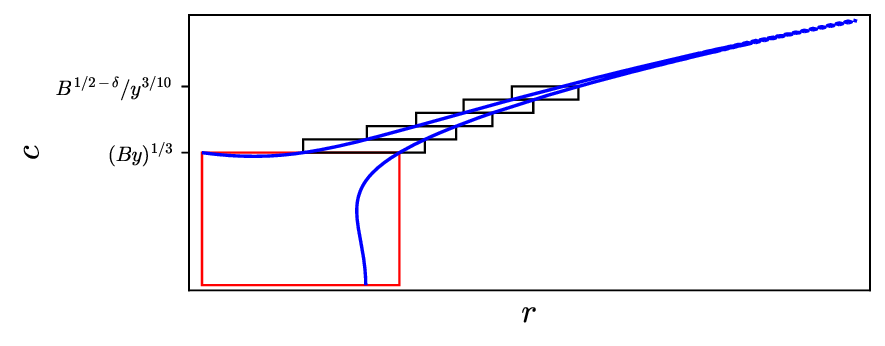}
    \caption{The graph of $(r+c^2)\sqrt{|r-c^2|}=8$ is in blue. Our goal is to obtain a weighted count of lattice points in certain congruence classes within regions of this shape. In the first case with $c\leq (By)^{1/3}$, we will obtain an upper bound through considering the lattice points within the red rectangle. In our second case with $(By)^{1/3}\leq c\leq B^{1/2-\delta}/y^{3/10}$, we will obtain an upper bound through considering lattice points in the black rectangles. Finally, for the last case where $B^{1/2-\delta}/y^{3/10}\leq c$, we will use a trivial estimate.}
    \label{fig:r_and_c}
\end{figure} 
Our goal this section is to prove Proposition \ref{prop:sum_of_Sy_for_y_big}. We first introduce a new set $S_y'$, into which $S_y$ will inject. First, note that $$S_{y}=\{(a,b,c):\max(|a|,|b|,|c|)^2 \sqrt{|\sqf(a^2+b^2-c^2)|}\leq B,\sqf(a^2+b^2-c^2) \notin \{0,1\}, \sq(a^2+b^2-c^2)=y\}$$ is in bijection with $$\{(a,b,c):\max(|a|,|b|,|c|)^2 \sqrt{|a^2+b^2-c^2|}\leq By, y=\sq(a^2+b^2-c^2), \sqf(a^2+b^2-c^2) \notin \{0,1\}\}.$$ Next, since we're looking for an upper bound, we may replace $\max(|a|,|b|,|c|)^2$ with $\frac{1}{3}(a^2+b^2+c^2)$. We may also replace the constraint $y=\sq(a^2+b^2-c^2)$ with the looser requirement that $y^2|a^2+b^2-c^2$ and $a^2+b^2-c^2 \neq 0$. Thus, $S_y$ injects into $$\{(a,b,c):(a^2+b^2+c^2)\sqrt{|a^2+b^2-c^2|}\leq 3By, y^2|a^2+b^2-c^2, a^2+b^2-c^2 \neq 0\}.$$ We 
define $$S_y':=\{(a,b,c):(a^2+b^2+c^2)\sqrt{|a^2+b^2-c^2|}\leq By, y^2|a^2+b^2-c^2, a^2+b^2-c^2 \neq 0\},$$ which is essentially the above set, except the $3$ is removed. We will proceed to show that $$\sum_{1 \leq y \leq B^{5/19-\epsilon}}|S_y'| = O_{\epsilon}(B \log B),$$ which will show that $$\sum_{1 \leq y \leq B^{5/19-\epsilon}}|S_y| = O_{\epsilon}((3B) \log (3B))=O_{\epsilon}(B \log B).$$ Let $r=a^2+b^2$, then let $$R_y=\{(r,c):(r+c^2)\sqrt{|r-c^2|}\leq By, y^2|r-c^2, r-c^2 \neq 0\}.$$ Then, we have $$|S_y'| = 4\sum_{(r,c) \in R_y} F(r),$$ where we recall $F(r)$ is the number of ways to write $r$ as a sum of two positive squares. So, we will bound this sum. Note that $c$ is at most $B^{1/2}$. So, we will write this sum as $$\sum_{c=1}^{B^{1/2}} \sum_{(r,c) \in R_y}F(r).$$ 

Now, having fixed $c$, the inner sum $\sum_{(r,c) \in R_y}F(r)$ is a sum over integers $r \equiv c^2 \mod y^2$ within a certain range. For example, if $c\leq (By)^{1/3}$, then the constraint $(r+c^2)\sqrt{|r-c^2|} \leq By$ is implied by $r \leq 2(By)^{2/3}$. The region of such $r,c$ is pictorially represented by the red box in Figure \ref{fig:r_and_c}. So, fixing a $c \leq (By)^{1/3}$, we would have that $$\sum_{(r,c) \in R_y}F(r) \leq \sum_{r\leq 2(By)^{2/3}, r \equiv c^2 \mod y^2} F(r).$$ By the squareroot law, one would guess that $$\sum_{r\leq (By)^{2/3}, r \equiv c^2 \mod y^2} F(r)=M(c,y)\frac{2(By)^{2/3}}{y^2}+E(c)$$ where $M(c,y)$ is some constant and $E(c)$ would be somewhere on the order of $\frac{\sqrt{2}(By)^{1/3}}{y}$. However, such error estimates are only conjectured \cite[Conj. 13.9]{montgomery_vaughan_2007} and it would turn out that applying the Tauberian theorem to the appropriate L-function would only give us $E=O_{\epsilon}((By)^{1/3+\epsilon})$, which turns out to not be good enough once we sum over $c$ and $y$. Despite this, we do have an averaged version of our conjectured error term. That is, we can obtain a good estimate of $\sum_{c=M}^{N} E(c)$, when $N-M$ is big. Thus, in the first two cases, we upper bound our sum with multiple rectangles as in Figure \ref{fig:r_and_c}.

Our strategy is as follows.
\begin{enumerate}
    \item Give an estimate of the sum $\sum_m\sum_{r \equiv m^2 \mod y^2}F(r)$ where $m,r$ vary in some large intervals as a main term plus an averaged error term. Note that the indices in this sum correspond to the lattice points in the red and black boxes in Figure \ref{fig:r_and_c}. In practice, we will place slightly more restrictions on $m$. 
    \item Apply the bounds to the first and second cases as in Figure \ref{fig:r_and_c}.
    \item Finally, we have a trivial estimate for the third case of Figure \ref{fig:r_and_c}, and conclude.
\end{enumerate}

We start with the first step, and attempt to obtain a good averaged error term. First we shall introduce some notations. Fix a residue class $c^2 \mod y^2$. Let $l=(c,y)$. Let $\widehat{(\bbZ/(y^2/l^2)\bbZ)^\times}$ be the characters of the multiplicative group $(\bbZ/(y^2/l^2)\bbZ)^\times$. We have that \begin{align*}
    \sum_{r=1, r \equiv c^2 \mod y^2}^{X} F(r)&=\sum_{r=1, r \equiv (c/l)^2 \mod (y/l)^2}^{X/l^2} F(rl^2)\\
    &=\frac{1}{\phi(y^2/l^2)}\sum_{\chi \in \widehat{(\bbZ/(y^2/l^2)\bbZ)^\times}}\chi((c/l)^2)\sum_{r=1}^{X/l^2} F(rl^2)\overline{\chi}(r).
\end{align*} Now, fix $\chi \in \widehat{(\bbZ/(y^2/l^2)\bbZ)^\times}$ and consider $$L^l(\chi,s)=\sum_{i=2v_2(l)}^{\infty} \frac{\chi(2^{i-2v_2(l)})}{2^{is}}\prod_{p \equiv 1 \mod 4} \sum_{i=2v_p(l)}^{\infty} \frac{(i+1)\chi(p^{i-2v_p(l)})}{p^{is}}\prod_{p \equiv 3 \mod 4} \sum_{i=v_p(l)}^{\infty} \frac{\chi(p^{2i-2v_p(l)})}{p^{2is}}.$$ Let $a^l_{\chi,n}$ be the $n$-th coefficient of this $L$-function. Then, $$a^l_{\chi,n} = \begin{cases}
F(n) \chi(n/l^2) & l^2 | n \\
0 & l^2 \nmid n
\end{cases}.$$ Thus, $\sum_{r=1}^{X/l^2} F(rl^2)\overline{\chi}(r)=\sum_{n=1}^{X}a^l_{\chi,n}$. 
Consider the Hecke character lift $\chi_{\adelic}$ of $\chi$ defined by $\chi_{\adelic} = \chi_{\infty} \prod_{\pi \in \Spec(\bbZ[i])} \chi_{\pi}$, where $\chi_{\pi}$'s and $\chi_{\infty}$ are defined as follows. Define $f_{\pi} \in \bbZ$ so that $y^2/l^2$ factorizes as $\prod_{\pi} \pi^{f_{\pi}}$ in $\bbZ[i]$. Then, given $x_{\pi} \in \bbZ[i]_{\pi}$, define $$\chi_{\pi}(x_{\pi}) = \begin{cases}
    \chi(N(\pi))^m & x_{\pi} \in \pi^m \bbZ[i]_{\pi}, \pi \nmid y^2/l^2 \\
    \chi(N(j)) & x_{\pi} \in \pi^k(j+\pi^{f_{\pi}}\bbZ[i]_{\pi}), j \in \bbZ[i], k \in \bbZ, \pi \nmid j, \pi | y^2/l^2.
\end{cases}$$ For, $\chi_{\infty}$, we define it to be $$\chi_{\infty}(x_{\infty}) = \begin{cases}
    1 & \chi(-1) = 1\\
    1 & \chi(-1) = -1, N(x_{\infty}) > 0 \\
    -1 & \chi(-1) = -1, N(x_{\infty}) < 0 \\
\end{cases}.$$


With this definition, we see that $L^l(\chi,s)$ is just a modification of the Hecke L-function $L_{\bbQ[i]}(\chi_{\adelic},s)$. We have $$L^l(\chi,s)=L_{\bbQ[i]}(\chi_{\adelic},s) l^{-2s}\prod_{p|l, p \equiv 1 \mod 4}Q_{\chi,l,p}(s),$$ where $$Q_{\chi,l,p}(s)=
\frac{\sum_{i=0}^{\infty} (2v_p(l)+1+i)\chi(p^{i})/p^{is}}{\sum_{i=0}^{\infty} (i+1)\chi(p^{i})/p^{is}}.$$ With this form of $L^l(\chi,s)$, we may obtain an estimate on the sum $\sum_{n=1}^X a^l_{\chi,n}$ via Perron's formula. Note that we cannot use the Tauberian theorem as this $L$-function does not have positive coefficients for nontrivial $\chi$.

Our step one is to prove the following averaged error. \begin{prop}
\label{prop:averaged_error} Assuming GRH, for all $\epsilon>0$, we have
$$\sum_{m=M,(m,(y/l))=1}^{N} \left \lvert\sum_{\substack{r=1\\ r \equiv m^2 \mod (y/l)^2}}^{X/l^2} F(rl^2)-\frac{X}{\phi((y/l)^2)}\Res_{s=1}L^l(\chi_{\triv},s)\right\lvert$$$$=\begin{cases}
O(\frac{1}{l}(N-M)^{1/2}y^{\epsilon}X^{1/2+\epsilon}) & N-M \leq (y/l)^2 \\
O(\frac{N-M}{y}y^{\epsilon}X^{1/2+\epsilon}) & \text{ otherwise}
\end{cases}$$
\end{prop} This will be obtained through a technical lemma applied to estimates on the sums of coefficients of $L^l(\chi,s)$. One should compare this to the ``unaveraged" error term $$O(\frac{N-M}{l}y^{\epsilon}X^{1/2+\epsilon}),$$ which can be obtained by directly applying the Tauberian theorem. With the averaged error term, one saves a factor of $\min(y/l,\sqrt{N-M})$, which will be crucial in the case where $l$ is small.

\begin{lemma}
\label{lem:L^l_critical_line_bound} Let $l|y$. Let $\chi \in \widehat{(\bbZ/(y/l)^2\bbZ)}$ be any character. Assuming GRH, for all $\epsilon>0$,
$$L^l(\chi,1/2+it) = O_{\epsilon}((|s|y)^{\epsilon}l^{-1})$$
\end{lemma}
\begin{proof}
First we have $$L^l(\chi,s)=L_{\bbQ[i]}(\chi_{\adelic},s) l^{-2s}\prod_{p|l, p \equiv 1 \mod 4}Q_{\chi,l,p}(s).$$ For $L_{\bbQ[i]}(\chi_{\adelic},s)$, we may use the GRH to obtain a bound of $(\frac{y|s|}{l})^{\epsilon}$ on the line $\Re(s)=1/2$. For $\prod_{p|l}Q_{\chi,l,p}(s)$, we have that for $s$ on the line $\Re(s) = 1/2$, \begin{align*}
    |Q_{\chi,l,p}(s)| &= \left| \frac{\sum_{i=0}^{\infty} (2v_p(l)+1+i)\chi(p^{i})/p^{is}}{\sum_{i=0}^{\infty} (i+1)\chi(p^{i})/p^{is}}\right| \\
    &=\left|(1-\chi(p)p^{-s})(1-\overline{\chi(p)}p^{-s})\sum_{i=0}^{\infty} (2v_p(l)+1+i)\chi(p^{i})/p^{is}\right| \\
    &\leq \frac{2v_p(l)+1}{(1+p^{-1/2})^2}\sum_{i=0}^{\infty}(i+1)p^{-i/2}\\
    &= (2v_p(l)+1) \frac{(1+p^{-1/2})^2}{(1-p^{-1/2})^2} \leq \frac{2v_p(l)+1}{(1-p^{-1/2})^4}
\end{align*}
Thus, for $s$ on the line $\Re(s) = 1/2$, $$\left|\prod_{p|l, p \equiv 1 \mod 4} Q_{\chi,l,p}(s)\right| \leq \prod_{p|l, p \equiv 1 \mod 4} \frac{2v_p(l)+1}{(1-p^{-1/2})^4} = O_{\epsilon}(l^{\epsilon}).$$ Thus, combining these bounds, and noting that $l \leq y$, we obtain our desired bound $(y|s|)^{\epsilon}l^{-1}$.
\end{proof}

Applying the above lemmas gives us an estimate on the sum of coefficients of $L^l(\chi,s)$, from which we obtain the following lemma. The proof of this is quite standard. However, we write the proof out to keep track of the error terms more closely. Specifically, we need to keep track of the error terms as $l$ and $y$ vary.
\begin{lemma}
\label{lem:sum_of_f(r)x(r)} Assuming GRH, for all $\epsilon>0$, 
$$\sum_{r=1}^{X/l^2} F(rl^2)\overline{\chi}(r)=E_0(\chi)X+O_{\epsilon}(l^{-1}|y|^{\epsilon}X^{1/2+\epsilon}),$$ where $$E_0(\chi)=\begin{cases}
\Res_{s=1} L^l(\chi_{\triv},s) & \chi = \chi_{\triv}\\
0 & \chi \neq \chi_{\triv}.
\end{cases}$$
\end{lemma}
\begin{proof}

Recall that $\sum_{r=1}^{X/l^2} F(rl^2)\overline{\chi}(r)=\sum_{n=1}^{X}a^l_{\chi,n}$. So, we shall estimate the latter sum. For the rest of this proof, let $a_n:=a^l_{\chi,n}$. Let $X'$ be the closest half-integer to $X$. Perron's formula \cite[Theorem 5.2]{montgomery_vaughan_2007} gives us that $$\sum_{n=1}^{X'} a_n = \frac{1}{2\pi i} \int_{2-iT}^{2+iT}L^l(\chi,s) \frac{X'^s}{s} ds + R,$$ where $$R=\frac{1}{\pi} \sum_{X'/2 < n < X'} a_n \min(1,\frac{1}{T \log \frac{X'}{n}}) - \frac{1}{\pi} \sum_{X'<n<2X'} a_n \min(1,\frac{1}{T \log \frac{n}{X'}}) + O(\frac{X'^2}{T}\sum_{n} \frac{|a_n|}{n^2}).$$ Since we took $X'$ to be a half-integer $\log \frac{X'}{n}$ and $\frac{n}{X'}$ will be bounded from below by $$\log(\frac{X'}{X'-1/2}) \gg \frac{1}{X'}.$$ Now, $a_n = O_{\epsilon}(n^{\epsilon})$, so after choosing $T=X'^3$, we may bound $R$ by $O(1/X')$.

Now it comes down to estimating $$\sum_{n=1}^{X'} a_n = \frac{1}{2\pi i} \int_{2-iT}^{2+iT}L^l(\chi,s) \frac{X'^s}{s} ds + R,$$ where $T=X'^3$. For this, we use the GRH, and push the contour of integration to the left. We will use the contour outlined by the four points $2-iT, 1/2-iT, 1/2+iT, 2+iT$. Now if $\chi=\chi_{\triv}$, we pick up a pole at $s=1$ of residue $X'\Res_{s=1} L^l(\chi_{\triv},s)$. Otherwise, there is no pole.

The rest of the estimate comes from the contour. By the Phragm\'en-Lindel\"of Principle, plus the estimate based on the GRH that $L^l(\chi,1/2+it) = O_{\epsilon}(|yt|^{\epsilon})$, we have $$\left\lvert\int_{2-iT}^{1/2-iT} L^l(\chi,s) \frac{X'^s}{s}ds\right\lvert \leq |y|^{\epsilon}|T|^{\epsilon}\int_{2-iT}^{1/2-iT} \left\lvert\frac{X'^s}{s}\right\lvert ds\leq |y|^{\epsilon}|T|^{\epsilon}\frac{X'^2}{T}=O_{\epsilon}(|y|^{\epsilon}).$$ Similarly, we also have $$\left\lvert \int_{2+iT}^{1/2+iT} L^l(\chi,s) \frac{X'^s}{s}ds\right\lvert = O_{\epsilon}(y^{\epsilon}).$$

The only part left of the contour is $1/2-iT$ to $1/2+iT$. Here we use Lemma \ref{lem:L^l_critical_line_bound}, and obtain $$\left\lvert\int_{1/2-iT}^{1/2+iT} L^l(\chi,s) \frac{X'^s}{s}ds\right\lvert = O(l^{-1}|y|^{\epsilon}T^{\epsilon}\int_{1/2-iT}^{1/2+iT} \left \lvert \frac{X'^s}{s}\right\lvert ds)=O(l^{-1}|y|^{\epsilon}X'^{1/2}T^{2\epsilon})=O(l^{-1}|y|^{\epsilon}X'^{1/2+\epsilon}).$$ After plugging in $X'=X+O(1)$, we're done.
\end{proof}
We also have the following technical lemma.
\begin{lemma} \label{lem:temp131313}
Let $f:\widehat{(\bbZ/q\bbZ)^\times} \to \bbC$ be any function. Then, $$\sum_{m=1,(m,q)=1}^q \left\lvert\sum_{\chi}f(\chi)\chi(m)\right\lvert^2=\phi(q)\sum_{\chi}|f(\chi)|^2.$$
\end{lemma}
\begin{proof}
\begin{align*}
    \sum_{m=1,(m,q)=1}^q \left\lvert\sum_{\chi}f(\chi)\chi(m)\right\lvert^2&=\sum_{m=1,(m,q)=1}^q \sum_{\chi}\sum_{\chi'}f(\chi)\chi(m)\overline{f(\chi')\chi'(m)} \\
    &= \sum_{\chi}\sum_{\chi'}f(\chi)\overline{f(\chi')}\sum_{m=1,(m,q)=1}^q \chi(m)\overline{\chi'(m)}\\
    &=\sum_{\chi}\sum_{\chi'}f(\chi)\overline{f(\chi')} \begin{cases}\phi(q) & \chi=\chi' \\ 0 & \chi \neq \chi'\end{cases}\\
    &=\phi(q)\sum_{\chi}|f(\chi)|^2
\end{align*}
\end{proof}

Now we may prove the averaged error.
\begin{proof}[Proof of Proposition \ref{prop:averaged_error}]

First, we have Cauchy-Schwarz giving us that \begin{align*}
    &\sum_{\substack{m=M\\(m,(y/l)^2)=1}}^{N} \left(\sum_{\substack{r=1\\ r \equiv m^2 \mod (y/l)^2}}^{X/l^2} F(rl^2)-\frac{X}{\phi((y/l)^2)}\Res_{s=1}L^l(\chi_{\triv},s)\right) \\
    &\leq (N-M)^{1/2}\sqrt{\sum_{\substack{m=M\\(m,(y/l)^2)=1}}^{N} \left\lvert\sum_{\substack{r=1\\ r \equiv m^2 \mod (y/l)^2}}^{X/l^2} F(rl^2)-\frac{X}{\phi((y/l)^2)}\Res_{s=1}L^l(\chi_{\triv},s)\right\lvert^2}\\
    &\leq (N-M)^{1/2}\sqrt{\bigg\lceil \frac{N-M}{(y/l)^2} \bigg\rceil\sum_{\substack{m=1\\(m,(y/l)^2)=1}}^{(y/l)^2} \left\lvert\sum_{\substack{r=1\\ r \equiv m^2 \mod (y/l)^2}}^{X/l^2} F(rl^2)-\frac{X}{\phi((y/l)^2)}\Res_{s=1}L^l(\chi_{\triv},s)\right\lvert^2}\\
    &=(N-M)^{1/2}\bigg\lceil \frac{N-M}{(y/l)^2} \bigg\rceil^{1/2}\sqrt{\sum_{\substack{m=1\\(m,(y/l)^2)=1}}^{(y/l)^2} \left\lvert\sum_{\substack{r=1\\ r \equiv m^2 \mod (y/l)^2}}^{X/l^2} F(rl^2)-\frac{X}{\phi((y/l)^2)}\Res_{s=1}L^l(\chi_{\triv},s)\right\lvert^2}
\end{align*} Now $$(\bbZ/(y/l)^2\bbZ)^\times\cong \prod_{p|(y/l)^2} (\bbZ/p^{v_p((y/l)^2)}\bbZ)^\times.$$ Thus, by the structure of unit groups of the form $(\bbZ/p^i\bbZ)^\times$, we have that as $m$ runs between $1$ and $(y/l)^2$ with $(m,(y/l)^2)=1$, the number of times $m^2 \mod (y/l)^2$ hits any residue class $\mod (y/l)^2$ is at most $2^{\omega(y/l)}$. Thus, we have the following continued estimate $$\leq (N-M)^{1/2}\bigg\lceil \frac{N-M}{(y/l)^2} \bigg\rceil^{1/2}\sqrt{2^{\omega(y/l)}\sum_{\substack{m=1\\(m,(y/l)^2)=1}}^{(y/l)^2} \left\lvert\sum_{\substack{r=1\\ r \equiv m \mod (y/l)^2}}^{X/l^2} F(rl^2)-\frac{X}{\phi((y/l)^2)}\Res_{s=1}L^l(\chi_{\triv},s)\right\lvert^2}$$\begin{equation}
\label{eqn:temp2}
    \leq (N-M)^{1/2}\bigg\lceil \frac{N-M}{(y/l)^2} \bigg\rceil^{1/2}y^{\epsilon}\sqrt{\sum_{\substack{m=1\\(m,(y/l)^2)=1}}^{(y/l)^2} \left\lvert\sum_{\substack{r=1\\ r \equiv m \mod (y/l)^2}}^{X/l^2} F(rl^2)-\frac{X}{\phi((y/l)^2)}\Res_{s=1}L^l(\chi_{\triv},s)\right\lvert^2},
\end{equation} using the fact that $2^{\omega(y/l)}=O(y^{\epsilon})$.

Next, apply Lemma \ref{lem:temp131313} to $$f(\chi)=\frac{1}{\phi((y/l)^2)} \left(\sum_{r=1}^{X/l^2} F(rl^2)\overline{\chi}(r)-E_0(\chi)X\right).$$ This will show that \begin{align*}
    &\sum_{\substack{m=1\\(m,(y/l)^2)=1}}^{(y/l)^2} \left(\sum_{\substack{r=1\\ r \equiv m^2 \mod (y/l)^2}}^{X/l^2} F(rl^2)-\frac{X}{\phi((y/l)^2)}\Res_{s=1}L^l(\chi_{\triv},s)\right)^2\\
    &=\frac{1}{\phi((y/l)^2)}\sum_{\chi} \left\lvert\sum_{r=1}^{X/l^2} F(rl^2)\overline{\chi}(r)-E_0(\chi)X\right\lvert^2\\
    &=\frac{1}{\phi((y/l)^2)}\sum_{\chi}O(l^{-2}y^{2\epsilon}X^{1+2\epsilon})\\
    &= O(l^{-2}y^{2\epsilon}X^{1+2\epsilon}).
\end{align*} where the second to last inequality is by Lemma \ref{lem:sum_of_f(r)x(r)}. Thus, plugging this into Equation \ref{eqn:temp2}, we get the bound $$(N-M)^{1/2}\bigg\lceil \frac{N-M}{(y/l)^2} \bigg\rceil^{1/2}y^{\epsilon}X^{1/2+\epsilon}l^{-1}=\begin{cases}
O(l^{-1}(N-M)^{1/2}y^{2\epsilon}X^{1/2+\epsilon}) & N-M \leq (y/l)^2 \\
O(\frac{N-M}{y}y^{2\epsilon}X^{1/2+\epsilon}) & \text{ otherwise}
\end{cases}$$
\end{proof}

For the next steps, we need a bound on the main term, specifically $\Res_{s=1} L^l(\chi_{\triv},s)$. 
\begin{lemma}
\label{lem:L^l_residue} Let $\chi_{\triv} \in \widehat{(\bbZ/(y/l)^2\bbZ)}$ be the trivial character. Then,
$$\Res_{s=1} L^l(\chi_{\triv},s)=O(\frac{\prod_{p|l, p \equiv 1 \mod 4} (2v_p(l)+1)}{l^2})$$
\end{lemma}
\begin{proof}

Since $$L^l(\chi_{\triv},s) = \zeta_{\bbQ[i]}(s) l^{-2s} \prod_{p|l, p \equiv 1 \mod 4} Q_{\chi_{\triv},l,p}(s),$$ the residue at $s=1$ is $$(\Res_{s=1}\zeta_{Q[i]}(s))\prod_{p|l, p \equiv 1 \mod 4}Q_{\chi_{\triv},l,p}(1)$$ which by the class number formula is $$\frac{\pi}{l^2} \prod_{p|l, p \equiv 1 \mod 4}Q_{\chi_{\triv},l,p}(1).$$ Finally, we have \begin{align*}
    |Q_{\chi_{triv},l,p}(1)| &= \left| \frac{\sum_{i=0}^{\infty} (2v_p(l)+1+i)\chi_{triv}(p^{i})/p^{i}}{\sum_{i=0}^{\infty} (i+1)\chi_{triv}(p^{i})/p^{i}}\right| \\
    &= \frac{\left|\sum_{i=0}^{\infty} (2v_p(l)+1+i)\chi_{triv}(p^{i})/p^{i}\right|}{\left|\sum_{i=0}^{\infty} (i+1)\chi_{triv}(p^{i})/p^{i}\right|} \\
    &\leq (2v_p(l)+1)\frac{\left|\sum_{i=0}^{\infty} (i+1)\chi_{triv}(p^{i})/p^{i}\right|}{\left|\sum_{i=0}^{\infty} (i+1)\chi_{triv}(p^{i})/p^{i}\right|} \\
    &= 2v_p(l)+1    
\end{align*}
Taking the product over $p|l$ and $p \equiv 1 \mod 4$, we obtain our bound.

\end{proof}

We will also need the following technical lemma.
\begin{lemma}[Erdos-Turan inequality]
\label{lem:erdos_turan}
Fix $y$ and $l|y$. Let $i,j \in [0,y)$ be any real number. Then, $$\# \{c \in [i,j] : l|c, (c/l,y/l)=1\} = \frac{\phi(y/l)}{y}(j-i)+O(\sqrt{\phi(y/l)}).$$
\end{lemma}
\begin{proof}
This is an application of the Erdos-Turan inequality to the probability measure $\mu$ on $[0,1)$ given by $\mu(S)=\frac{1}{\phi(y/l)}\#\{x \in S: yx \in \bbZ, l|yx, (yx/l,y/l)=1\}$.
\end{proof}

\subsection{$c\leq (By)^{1/3}$}

In this subsection, we wish to bound $$\sum_{(r,c) \in R_y, c\leq (By)^{1/3}} F(r).$$ Fix $c \leq  (By)^{1/3}$. Then, the bound $(r+c^2)\sqrt{|r-c^2|}\leq By$ implies that $r \leq  2(By)^{2/3}$. Thus, $$R_y \cap \{(r,c)|c\leq B^{1/3}\} \subset \{(r,c)|c\leq B^{1/3},r\leq B^{2/3}, y^2|r-c^2, r-c^2 \neq 0\}.$$ So, $$\sum_{(r,c) \in R_y, c\leq (By)^{1/3}} F(r) \leq \sum_{c=1}^{(By)^{1/3}}\sum_{r=1, r \equiv c^2 \mod y^2}^{2(By)^{2/3}} F(r).$$ 

\begin{lem}
\label{lem:mid_y_estimate_small_c}
Assuming GRH, for any $\epsilon>0$, we have $$\sum_{c=1}^{(By)^{1/3}} \sum_{r=1, r \equiv c^2 \mod y^2}^{2(By)^{2/3}} F(r) = O_{\epsilon}\left(\frac{B}{y}\sum_{l|y}\frac{\prod_{p|l}(2v_p(l)+1)}{l}+B^{2/3+\epsilon}y^{-1/3+\epsilon} + (By)^{1/2+\epsilon}\right).$$
\end{lem}
\begin{proof}

First we split our sum by $(c,y)$. $$\sum_{c=1}^{(By)^{1/3}} \sum_{r=1, r \equiv c^2 \mod y^2}^{2(By)^{2/3}} F(r) = \sum_{l|y}\sum_{c=1,(c,y)=l}^{(By)^{1/3}}\sum_{r=1, r \equiv c^2 \mod y^2}^{2(By)^{2/3}} F(r)=\sum_{l|y}\sum_{c=1,(c,y)=l}^{(By)^{1/3}} \sum_{\substack{r=1 \\ r \equiv (c/l)^2 \mod (y/l)^2}}^{2(By)^{2/3}/l^2} F(rl^2) $$ Now we may apply Proposition \ref{prop:averaged_error} and obtain $$\sum_{l|y}\sum_{c=1,(c,y)=l}^{(By)^{1/3}} \frac{2(By)^{2/3}}{\phi(y^2/l^2)}\Res_{s=1}L^l(\chi_{\triv},s)+\sum_{l|y}\begin{cases}
O(l^{-1}((By)^{1/3})^{1/2}y^{2\epsilon}(By)^{1/3+\epsilon}) & (By)^{1/3} \leq (y/l)^2 \\
O(\frac{(By)^{1/3}}{y}y^{2\epsilon}(By)^{1/3+\epsilon}) & \text{ otherwise}
\end{cases}$$$$=\sum_{l|y} \frac{2(By)^{2/3}}{\phi(y^2/l^2)}\Res_{s=1}L^l(\chi_{\triv},s)\sum_{c=1,(c,y)=l}^{(By)^{1/3}} 1+\sum_{l|y}\begin{cases}
O((By)^{1/2+\epsilon}l^{-1}) & (By)^{1/3} \leq (y/l)^2 \\
O(B^{2/3+\epsilon}y^{-1/3+\epsilon}) & \text{ otherwise}
\end{cases}$$$$=A+E$$ where $$A=\sum_{l|y} \frac{2(By)^{2/3}}{\phi(y^2/l^2)}\Res_{s=1}L^l(\chi_{\triv},s)\sum_{c=1,(c,y)=l}^{(By)^{1/3}} 1$$ and $$E=\sum_{l|y}\begin{cases}
O((By)^{1/2+\epsilon}l^{-1}) & (By)^{1/3} \leq (y/l)^2 \\
O(B^{2/3+\epsilon}y^{-1/3+\epsilon}) & \text{ otherwise}
\end{cases}.$$

We first find the bound for the main term $A$. By Lemma \ref{lem:L^l_residue}, we get a bound on $\Res_{s=1}L^l(\chi_{\triv},s)$ of $\frac{\prod_{p|l, p \equiv 1 \mod 4}(2v_p(l)+1)}{l^2}$, and by Lemma \ref{lem:erdos_turan}, we get that there are $(By)^{1/3}\frac{\phi(y/l)}{y}+O(\sqrt{\phi(y/l)})$ terms in the sum over $c$. Now in our case, we have $y < B^{5/19} < \sqrt{B}$, and so $B^{-1/3}y^{2/3} \leq 1 \leq \sqrt{\phi(y/l)}$. Rearranging the terms, we get $\sqrt{\phi(y/l)} \leq (By)^{1/3} \frac{\phi(y/l)}{y},$ and so the number of terms in the sum over $c$ is at most $O((By)^{1/3}\frac{\phi(y/l)}{y})$. Thus, we have that \begin{align*}
    A &=O\left(\sum_{l|y}\frac{2(By)^{2/3}}{\phi(y^2/l^2)}\frac{\prod_{p|l, p \equiv 1 \mod 4}(2v_p(l)+1)}{l^2}(By)^{1/3}\frac{\phi(y/l)}{y}\right)\\ 
    &= O\left(\frac{B}{y}\sum_{l|y}\frac{\prod_{p|l}(2v_p(l)+1)}{l}\right),
\end{align*}
where in the second equality, for the sake of an upper bound, we forget the restriction that $p \equiv 1 \mod 4$.


Now we estimate $E$. \begin{align*}
    E &= \sum_{l|y}\begin{cases}
    O((By)^{1/2+\epsilon}l^{-1}) & (By)^{1/3} \leq (y/l)^2 \\
    O(B^{2/3+\epsilon}y^{-1/3+\epsilon}) & \text{ otherwise}
    \end{cases} \\
    &= \sum_{l|y, l\leq B^{-1/6}y^{5/6}}
    O_{\epsilon}((By)^{1/2+\epsilon}l^{-1}) + \sum_{l|y, l>B^{-1/6}y^{5/6}}
    O_{\epsilon}(B^{2/3+\epsilon}y^{-1/3+\epsilon}) \\
    &= O_{\epsilon}((By)^{1/2+2\epsilon}+B^{2/3+\epsilon}y^{-1/3+2\epsilon})
\end{align*}
\end{proof}

\subsection{$(By)^{1/3}\leq c\leq \frac{B^{1/2-\delta}}{y^{3/10}}$}

In this subsection, we wish to bound $$\sum_{\substack{(r,c) \in R_y\\(By)^{1/3}\leq c\leq B^{1/2-\delta}/y^{3/10}}}F(r) .$$For $c$ in this range, we will bound the size of $r$. Recall that the set in question is $$R_y=\{(r,c):(r+c^2)\sqrt{|r-c^2|}\leq By, y^2|r-c^2, r-c^2 \neq 0\},$$ where we will delay our choice of $\delta$ until the end of this section.

\begin{lemma}
\label{lem:bound_on_r_based_on_c}
Fix any $c$. Suppose $r>c^2$. Then, $(r+c^2)\sqrt{|r-c^2|}\leq By$ implies that $r \geq c^2-\frac{(By)^2}{c^4}$ and $r\leq c^2+\frac{(By)^2}{c^4}$. 
\end{lemma}
\begin{proof}
For the lower bound, we have $$c^2\sqrt{c^2-r}\leq (r+c^2)\sqrt{|r-c^2|}\leq By.$$ Thus, $c^4(c^2-r)\leq (By)^2$. Thus, $r \geq c^2-\frac{(By)^2}{c^4}$

For the upper bound, we assume $r>c^2$. Thus, we have $2c^2\sqrt{r-c^2}\leq (r+c^2)\sqrt{|r-c^2|}\leq By$. Hence $4c^4(r-c^2)\leq (By)^2$. Thus, $r\leq c^2+\frac{(By)^2}{4c^4}$. In particular, $r\leq c^2+\frac{(By)^2}{c^4}$.
\end{proof}

From this lemma, we see that $$\sum_{\substack{(r,c) \in R_y\\(By)^{1/3}\leq c\leq B^{1/2-\delta}/y^{3/10}}}F(r) \leq \sum_{c=(By)^{1/3}}^{B^{1/2-\delta}/y^{3/10}}\sum_{\substack{r=c^2-\frac{(By)^2}{c^4}\\r \equiv c^2 \mod y^2}}^{c^2+\frac{(By)^2}{c^4}}F(r).$$ Thus, it will suffice to bound the double sum.

\begin{lemma}
\label{lem:mid_y_estimate_medium_c}
Assuming GRH, for any $\epsilon\geq\delta>0$, we have $$\sum_{c=(By)^{1/3}}^{B^{1/2-\delta}/y^{3/10}}\sum_{\substack{r=c^2-\frac{(By)^2}{c^4}\\r \equiv c^2 \mod y^2}}^{c^2+\frac{(By)^2}{c^4}}F(r)=O_{\epsilon,\delta}(\frac{B}{y}\sum_{l|y}\frac{1}{l}\prod_{p|l}(2v_p(l)+1)+B^{1-2\delta+\epsilon}y^{-1+3\epsilon}).$$
\end{lemma}
\begin{proof}

We again have to use Proposition \ref{prop:averaged_error} to obtain a good bound. We will split our summation over $c$ in two ways. We first split it by $(c,y)$, then we split into small intervals as in the black rectangles in Figure \ref{fig:r_and_c}. Let $n_i=(By)^{1/3}+i y^{4/5}$. Then, we have $$
    \sum_{c=(By)^{1/3}}^{B^{1/2-\delta}/y^{3/10}}\sum_{\substack{r=c^2-\frac{(By)^2}{c^4}\\r \equiv c^2 \mod y^2}}^{c^2+\frac{(By)^2}{c^4}}F(r) 
    = \sum_{l|y} \sum_{i=0}^{B^{1/2-\delta}/y^{11/10}-1}\sum_{c=n_i,(c,y)=l}^{n_{i+1}}\sum_{\substack{r=c^2-\frac{(By)^2}{c^4}\\r \equiv c^2 \mod y^2}}^{c^2+\frac{(By)^2}{c^4}}F(r)$$

Now for each interval, we upper bound the sum as follows. $$\sum_{c=n_i,(c,y)=l}^{n_{i+1}}\sum_{\substack{r=c^2-\frac{(By)^2}{c^4}\\r \equiv c^2 \mod y^2}}^{c^2+\frac{(By)^2}{c^4}}F(r) \leq \sum_{c=n_i,(c,y)=l}^{n_{i+1}}\sum_{\substack{r=n_i^2-\frac{(By)^2}{n_i^4}\\r \equiv c^2 \mod y^2}}^{n_{i+1}^2+\frac{(By)^2}{n_{i+1}^4}}F(r)=\sum_{c=n_i,(c,y)=l}^{n_{i+1}}\sum_{\substack{r=(n_i^2-\frac{(By)^2}{n_i^4})/l^2\\r \equiv (c/l)^2 \mod (y/l)^2}}^{(n_{i+1}^2+\frac{(By)^2}{n_{i+1}^4})/l^2}F(rl^2).$$ We may use Corollary $\ref{prop:averaged_error}$ on this sum, and get $$\sum_{c=n_i,(c,y)=l}^{n_{i+1}}\frac{1}{\phi(y^2/l^2)}\Res_{s=1}L^l(\chi_{\triv},s)((n_{i+1}^2+\frac{(By)^2}{n_{i+1}^4})-(n_{i}^2+\frac{(By)^2}{n_{i}^4}))$$$$+\begin{cases}
O(\frac{1}{l}(y^{4/5})^{1/2}y^{\epsilon}(n_{i+1}^2+\frac{(By)^2}{n_{i+1}^4})^{1/2+\epsilon}) & y^{4/5} \leq (y/l)^2 \\
O(\frac{y^{4/5}}{y}y^{\epsilon}(n_{i+1}^2+\frac{(By)^2}{n_{i+1}^4})^{1/2+\epsilon}) & \text{ otherwise}
\end{cases}=A_{i,l}+E_{i,l}$$ where $$A_{i,l} := \sum_{c=n_i,(c,y)=l}^{n_{i+1}}\frac{1}{\phi(y^2/l^2)}\Res_{s=1}L^l(\chi_{\triv},s)((n_{i+1}^2+\frac{(By)^2}{n_{i+1}^4})-(n_{i}^2+\frac{(By)^2}{n_{i}^4}))$$ and $$E_{i,l} = \begin{cases}
O(\frac{1}{l}(y^{4/5})^{1/2}y^{\epsilon}(n_{i+1}^2+\frac{(By)^2}{n_{i+1}^4})^{1/2+\epsilon})) & y^{4/5} \leq (y/l)^2 \\
O(\frac{y^{4/5}}{y}y^{\epsilon}(n_{i+1}^2+\frac{(By)^2}{n_{i+1}^4})^{1/2+\epsilon}) & \text{ otherwise}
\end{cases}$$ $\chi_{\triv}$ is the trivial character in $\widehat{(\bbZ/(y/l)^2\bbZ)^\times}$.

Now our main term will be $\sum_{l} \sum_i A_{i,l}$, and our error term will be $\sum_l \sum_i E_{i,l}$. Lemma \ref{lem:L^l_residue} gives us an estimate on the residue of $L^l(\chi_{\triv},s)$ and Lemma \ref{lem:erdos_turan} gives us an estimate on the number of terms in the summation over $c$. Thus, we obtain \begin{align*}
    A_{i,l}=O\Bigg(&\left((n_{i+1}-n_i)\frac{\phi(y/l)}{y}+\sqrt{\phi(y/l)}\right)\frac{1}{\phi(y^2/l^2)}\frac{\prod_{p|l, p \equiv 1 \mod 4}(2v_p(l)+1)}{l^2}\\
    &((n_{i+1}^2+\frac{(By)^2}{n_{i+1}^4})-(n_i^2-\frac{(By)^2}{n_i^4}))\Bigg).
\end{align*} For the sake of an upper bound, we may forget the restriction of $p \equiv 1 \mod 4$ in the product. Next, substituting $n_i=(By)^{1/3}+i y^{4/5}$ and simplifying, we obtain \begin{align*}
    A_{i,l}=O\Bigg(&\left(\frac{\phi(y/l)}{y^{1/5}}+\sqrt{\phi(y/l)}\right)\frac{1}{\phi(y^2/l^2)}\frac{\prod_{p|l}(2v_p(l)+1)}{l^2}\\
    &(y^{4/5}(By)^{1/3}+(i+1) y^{8/5}+\frac{(By)^2}{((By)^{1/3}+i y^{4/5})^4})\Bigg).
\end{align*} Summing over $i$, we get that \begin{align*}\sum_{i=0}^{B^{1/2-\delta}/y^{11/10}}A_{i,l}&=O\Bigg(\left(\frac{\phi(y/l)}{y^{1/5}}+\sqrt{\phi(y/l)}\right)\frac{1}{\phi(y^2/l^2)}\frac{\prod_{p|l}(2v_p(l)+1)}{l^2}\\
&((B^{1/2-\delta}/y^{11/10})y^{4/5}(By)^{1/3}+(B^{1/2-\delta}/y^{11/10})^2 y^{8/5}+\frac{(By)^2}{y^{4/5}((By)^{1/3})^3})\Bigg)\\
&=O\left(\left(\frac{\phi(y/l)}{y^{1/5}}+\sqrt{\phi(y/l)}\right)\frac{1}{\phi(y^2/l^2)}\frac{\prod_{p|l}(2v_p(l)+1)}{l^2}(B^{5/6}y^{1/30}+B^{1/2}y^{-3/5}+By^{1/5})\right)\\
&=O\left(\left(\frac{\phi(y/l)}{y^{1/5}}+\sqrt{\phi(y/l)}\right)\frac{1}{\phi(y^2/l^2)}\frac{\prod_{p|l}(2v_p(l)+1)}{l^2}By^{1/5}\right)\\
    &=O\left(\frac{B}{yl}\prod_{p|l}(2v_p(l)+1)+\frac{\sqrt{(y/l)\prod_{p|(y/l)}(1-1/p)}}{y^{4/5}}\frac{\prod_{p|l}(2v_p(l)+1)B}{y \prod_{p|(y/l)}(1-1/p)}\right) \\
    &=O\left(\frac{B}{yl}\prod_{p|l}(2v_p(l)+1)+\frac{B}{y^{13/10} l^{1/2}}\prod_{p|l}(2v_p(l)+1)\right) \\
\end{align*} It remains to sum over $l|y$. For the first term, we leave it as $$\frac{B}{y}\sum_{l|y}\frac{1}{l}\prod_{p|l}(2v_p(l)+1).$$ For the second term, we have $$\sum_{l|y} \frac{B}{y^{13/10} l^{1/2}}\prod_{p|l}(2v_p(l)+1)=O_{\epsilon}(\frac{B}{y^{13/10}}\sum_{l|y}l^{-1/2+\epsilon})=O_{\epsilon}(By^{-13/10+\epsilon}),$$ which is smaller than our first term. Thus, all in all, we obtain as our main term $$\sum_{l|y}\sum_{i=0}^{B^{1/2-\delta}/y^{11/10}}A_{i,l}=O(\frac{B}{y}\sum_{l|y}\frac{1}{l}\prod_{p|l}(2v_p(l)+1)).$$

We now estimate the error term. First we have the estimate \begin{align*}
    (n_{i+1}^2+\frac{(By)^2}{n_{i+1}^4})^{1/2} & \leq n_{i+1} + \frac{By}{n_{i+1}^2} \\
    &=(By)^{1/3}+i y^{4/5}+\frac{By}{((By)^{1/3}+i y^{4/5})^2}\\
    &\leq 2 (By)^{1/3}+ i y^{4/5}.
\end{align*} Similarly, it is quick to see that $(n_{i+1}^2+\frac{(By)^2}{n_{i+1}^4})^{\epsilon} \leq (By)^{\epsilon}$. So, \begin{align*}
    \sum_{l|y}\sum_{i=0}^{B^{1/2-\delta}/y^{11/10}}E_{i,l}&=\sum_{l|y,l\leq y^{3/5}}\sum_{i=0}^{B^{1/2-\delta}/y^{11/10}}E_{i,l}+\sum_{l|y,l>y^{3/5}}\sum_{i=0}^{B^{1/2-\delta}/y^{11/10}}E_{i,l} \\
    &= \sum_{l|y,l\leq y^{3/5}}\sum_{i=0}^{B^{1/2-\delta}/y^{11/10}}O_{\epsilon}(l^{-1}y^{2/5+\epsilon}(n_{i+1}^2+\frac{(By)^2}{n_{i+1}^4})^{1/2+\epsilon})\\
    &+\sum_{l|y,l>y^{3/5}}\sum_{i=0}^{B^{1/2-\delta}/y^{11/10}}O_{\epsilon}(y^{-1/5+\epsilon}(n_{i+1}^2+\frac{(By)^2}{n_{i+1}^4})^{1/2+\epsilon})\\
    &= \sum_{l|y,l\leq y^{3/5}}\sum_{i=0}^{B^{1/2-\delta}/y^{11/10}}O_{\epsilon}(l^{-1}y^{2/5+\epsilon}((By)^{1/3}+i y^{4/5})(By)^{\epsilon})\\
    &+\sum_{l|y,l>y^{3/5}}\sum_{i=0}^{B^{1/2-\delta}/y^{11/10}}O_{\epsilon}(y^{-1/5+\epsilon}((By)^{1/3}+i y^{4/5})(By)^{\epsilon})\\
    &= \sum_{l|y,l\leq y^{3/5}}O_{\epsilon}(l^{-1}B^{\epsilon}y^{2/5+2\epsilon}((By)^{1/3}+\frac{B^{1/2-\delta}}{y^{3/10}})(\frac{B^{1/2-\delta}}{y^{11/10}}))\\
    &+\sum_{l|y,l>y^{3/5}}O_{\epsilon}(B^{\epsilon}y^{-1/5+2\epsilon}((By)^{1/3}+\frac{B^{1/2-\delta}}{y^{3/10}})(\frac{B^{1/2-\delta}}{y^{11/10}})) \\
\end{align*}
Now on the regime we are considering, $y \leq B^{5/19-\epsilon}$. This is exactly equivalent to $(By)^{1/3} \leq \frac{B^{1/2-\epsilon}}{y^{3/10}}$. Now, since we are choosing $\delta \leq \epsilon$, we see that $(By)^{1/3} \leq \frac{B^{1/2-\delta}}{y^{3/10}}$. Thus, \begin{align*}
    O_{\epsilon}(l^{-1}B^{\epsilon}y^{2/5+2\epsilon}((By)^{1/3}+\frac{B^{1/2-\delta}}{y^{3/10}})(\frac{B^{1/2-\delta}}{y^{11/10}})) &= O_{\epsilon}(l^{-1}B^{\epsilon}y^{2/5+2\epsilon}(\frac{B^{1/2-\delta}}{y^{3/10}})(\frac{B^{1/2-\delta}}{y^{11/10}}))\\
    &= O_{\epsilon}(l^{-1}B^{1-2\delta+\epsilon}y^{-1+2\epsilon}).
\end{align*} Similarly, $$O_{\epsilon}(B^{\epsilon}y^{-1/5+2\epsilon}((By)^{1/3}+\frac{B^{1/2-\delta}}{y^{3/10}})(\frac{B^{1/2-\delta}}{y^{11/10}}))=O_{\epsilon}(B^{1-2\delta+\epsilon}y^{-8/5+2\epsilon})$$
With this, we continue our bound and obtain \begin{align*}
    \sum_{l|y}\sum_{i=0}^{B^{1/2-\delta}/y^{11/10}}E_{i,l}&=\sum_{l|y,l\leq y^{3/5}}O_{\epsilon}(l^{-1}B^{1-2\delta+\epsilon}y^{-1+2\epsilon})+\sum_{l|y,l>y^{3/5}}O_{\epsilon}(B^{1-2\delta+\epsilon}y^{-8/5+2\epsilon})\\
    &= O_{\epsilon}(B^{1-2\delta+\epsilon}y^{-1+3\epsilon})
\end{align*}
\end{proof}

\subsection{$\frac{B^{1/2-\delta}}{y^{3/10}}\leq c\leq B^{1/2}$}

In this range, we may use a trivial estimate.

\begin{lemma}
\label{lem:mid_y_estimate_large_c}
Assuming GRH, for all $\epsilon,\delta>0$, $$\sum_{\substack{(r,c) \in R_y\\\frac{B^{1/2-\delta}}{y^{3/10}}\leq c\leq B^{1/2}}}F(r) =  O_{\epsilon}(B^{1/2+\epsilon+3\delta}y^{9/10})$$
\end{lemma}
\begin{proof}

\begin{align*}
    \sum_{\substack{(r,c) \in R_y\\\frac{B^{1-\delta}}{y^{3/10}}\leq c\leq B^{1/2}}}F(r) & \leq \sum_{\frac{B^{1/2-\delta}}{y^{3/10}}\leq c\leq B^{1/2}} \sum_{\substack{r=c^2-\frac{(By)^2}{c^4}\\r \equiv c^2 \mod y^2}}^{c^2+\frac{(By)^2}{c^4}}F(r)\\
    & \leq \sum_{\frac{B^{1/2-\delta}}{y^{3/10}}\leq c\leq B^{1/2}} \sum_{\substack{r=c^2-\frac{(By)^2}{c^4}\\r \equiv c^2 \mod y^2}}^{c^2+\frac{(By)^2}{c^4}}O_{\epsilon}(B^{\epsilon}) \\
    & \leq \sum_{\frac{B^{1/2-\delta}}{y^{3/10}}<c\leq B^{1/2}} O_{\epsilon}(B^{\epsilon}) \lceil \frac{B^2}{c^4} \rceil \\
    & =O_{\epsilon}(B^{1/2+\epsilon+3\delta}y^{9/10})
\end{align*}
\end{proof}

\subsection{Proof of Proposition \ref{prop:sum_of_Sy_for_y_big}}

Now we may prove our proposition.

\begin{proof}[Proof of Proposition \ref{prop:sum_of_Sy_for_y_big}]
By lemmas \ref{lem:mid_y_estimate_small_c}, \ref{lem:mid_y_estimate_medium_c}, and \ref{lem:mid_y_estimate_large_c}, we have that $$|S_y'| = O_{\epsilon_0,\delta}(\frac{B}{y}\sum_{l|y}\frac{1}{l}\prod_{p|l}(2v_p(l)+1)+B^{1-2\delta+\epsilon_0}y^{-1+3\epsilon_0} +(By)^{1/2+\epsilon_0}+B^{2/3+\epsilon_0}y^{-1/3+\epsilon_0}+B^{1/2+\epsilon_0+3\delta}y^{9/10}),$$ for any $\epsilon_0 \geq \delta>0$. We now make our choice of $\delta$ and $\epsilon_0$ to be $\epsilon/4$. Then, we may sum these terms over $y\leq B^{5/19-\epsilon}$ and obtain that all terms are $O(B)$ except for the first term, which we will now examine.

We want to bound $$\sum_{y=1}^{B^{5/19-\epsilon}}\frac{B}{y}\sum_{l|y}\frac{1}{l}\prod_{p|l}(2v_p(l)+1).$$ It suffices to show $$\sum_{y=1}^{B}\frac{1}{y}\sum_{l|y}\frac{1}{l}\prod_{p|l}(2v_p(l)+1)=O(\log B).$$ Now we have the bound \begin{align*}
    \sum_{y=1}^{B}\frac{1}{y}\sum_{l|y}\frac{1}{l}\prod_{p|l}(2v_p(l)+1) &\leq \sum_{y=1, \exists p \leq B, p | y}^{\infty}\frac{1}{y}\sum_{l|y}\frac{1}{l}\prod_{p|l}(2v_p(l)+1) \\
    &= \prod_{p \leq B}\left(1 + \frac{1+\frac{3}{p}}{p}+\frac{1+\frac{3}{p}+\frac{5}{p^2}}{p^2}+\dots\right).
\end{align*} Let $$M=1+\frac{3}{2}+\frac{5}{2^2}+\frac{7}{2^3}+\dots.$$ Then, the Euler product is bounded by \begin{align*}
    \prod_{p \leq B} \left(1+\frac{1}{p}+\frac{3}{p^2}+\frac{M}{p^2}+\frac{M}{p^3}+\dots\right) &= O(\prod_{p \leq B} (1+\frac{1}{p})) \\
    &= e^{\sum_{p \leq B}\log(1+1/p)+O(1)} \\
    &= e^{\sum_{p \leq B}\frac{1}{p} + O(1)} \\
    &= O(e^{\log \log B}) \\
    &= O(\log B).
\end{align*} 

\end{proof}

\section{Lower Bound} \label{sec:lower bound}

The goal of this section is to prove Proposition \ref{prop:thm_main_lower_bound}. Recall $$T := \{(a,b,c): \gcd(a,b) = 1, \max(|a|,|b|,|c|)^2 \sqrt{x}\leq B, a^2+b^2 \equiv 5 \mod 8,$$$$ \sqf(a^2+b^2-c^2) \equiv b \equiv 0 \mod 2\}.$$ We will lower bound $|T|$ by splitting it up into pieces, as follows. Define $T_{x}=\{(a,b,c) \in T: \sqf(a^2+b^2-c^2) = x\}$, so that $\bigsqcup_{x=2}^{B^{2/3-\epsilon}} T_x \subset T$ for any $\epsilon>0$. Also define $F'(n)$ to be the number of ways $n$ may be written as a sum of two coprime squares. Then, $$|T_x|=\frac{1}{2}\sum_{n=1}^{2B/\sqrt{x}}F'(n)G_x(n)$$ where $1/2$ comes from the fact that in every representation of $n$ as a sum of two coprime squares, one of them is even and one of them is odd but we require $b$ to be the even one.

To analyze $|T_x|$, we will consider the following $L$-function, $$L_{x,lower}=\sum_{n=1}^\infty \frac{a_n}{n^s},$$ where $a_n=\begin{cases}
F'(n)G_x(n) & n \equiv 5 \mod 8 \\
0 & \text{ otherwise}
\end{cases}$. Note that this $L$-function is extremely similar to $L_x(s)$, and so we will apply the same process to $L_{x,lower}$. Let $\widehat{(\bbZ/8\bbZ)^\times}$ denote the characters on the multiplicative group of $\bbZ/8\bbZ$. Then, $$L_{x,lower}(s)=\frac{1}{|(\bbZ/8\bbZ)^\times||\Cl(\bbQ[\sqrt{-x}])|}\sum_{\substack{\chi_1 \in \widehat{\Cl(\bbQ[\sqrt{-x}])} \\ \chi_2 \in \widehat{(\bbZ/8\bbZ)^\times}}}\chi_2(5) L_{0,lower}(\chi_1 \chi_2,s)$$ where \begin{align*}
    L_{0,lower}(\chi,s)=&\prod_{p \in P_{1,1}}\left(1+\frac{2\chi(\pi)}{p^{s}}+\frac{2\chi(\bar\pi)}{p^{s}}+\frac{2\chi(\pi^2)}{p^{2s}}+\frac{2\chi(\bar \pi^2)}{p^{2s}}+\dots\right)\\
    &\prod_{p \in P_{1,0}}\left(1+\frac{2\chi(p)}{p^{2s}}+\frac{2\chi(p^2)}{p^{4s}}+\dots\right)\\
    &\prod_{p \in P_{1,ram}}\left(1+\frac{2\chi(\pi)}{p^{s}}+\frac{2\chi(\pi^2)}{p^{2s}}+\dots\right).
\end{align*} Here we make the abuse of notation for $\chi_2$, so that $\chi_2(\pi)$ is defined by $\chi_2(N(\pi))$. 


Now we may estimate $L_{x,lower}(s)$ just as we did for $L_x(s)$. We give a meromorphic continuation with poles at $s=1$, we find the residue, and then we bound the size of $L_{x,lower}(s)$ near $\Re(s)=1/2$ in order to apply the Tauberian theorem. Instead of carrying out the same calculations again, we will state the relevant analogous lemmas.

Define $\chi_{\triv,\widehat{\Cl(\bbQ[\sqrt{-x}])}}$ to be the trivial character in $\widehat{\Cl(\bbQ[\sqrt{-x}])}$. Define $\chi_{\triv, \widehat{(\bbZ/8\bbZ)^\times}}$ to be the trivial character in $\widehat{(\bbZ/8\bbZ)^\times}$ and $\chi_{im, \widehat{(\bbZ/8\bbZ)^\times}}$ to be the ``imaginary" character (which takes values $1$ for $1,5 \mod 8$, and $-1$ for $3,7 \mod 8$). Also, recall $\chi_x$ was defined to be the character associated to the Jacobi Symbol $\chi_x(n)=\left(\frac{x}{n}\right)$.

\begin{lemma}
\label{lem:Lxlowers_meromorphic_continuation} $L_{x,lower}(s)$ may be meromorphically continued to $\Re(s)>1/2$, with a simple pole of $s=1$ of residue \begin{align*}
    \frac{1}{4|\Cl(\bbQ[\sqrt{-x}])|}&(\Res_{s=1}L_{0,lower}(\chi_{\triv,\widehat{\Cl(\bbQ[\sqrt{-x}])}}\chi_{\triv, \widehat{(\bbZ/8\bbZ)^\times}},s)\\
    &+\Res_{s=1}L_{0,lower}(\chi_{\triv,\widehat{\Cl(\bbQ[\sqrt{-x}])}}\chi_{im, \widehat{(\bbZ/8\bbZ)^\times}},s)).
\end{align*}
\end{lemma}
\begin{proof}
    This follows the same proof as Lemma \ref{lem:Lxs_meromorphic_continuation}.
\end{proof}
\begin{lemma} \label{lem:Lxlowers_residue}
Both $$\Res_{s=1}L_{0,lower}(\chi_{\triv,\widehat{\Cl(\bbQ[\sqrt{-x}])}}\chi_{\triv, \widehat{(\bbZ/8\bbZ)^\times}},s)$$ and $$\Res_{s=1}L_{0,lower}(\chi_{\triv,\widehat{\Cl(\bbQ[\sqrt{-x}])}}\chi_{im, \widehat{(\bbZ/8\bbZ)^\times}},s)$$ are $$\Omega(\frac{L(\chi_{x}\chi_{-1},1)L(\chi_x,1)}{\sqrt{x}})$$
\end{lemma}
\begin{proof}
    This follows the same proof as Lemma \ref{lem:Lxs_residue}, noting that for the sake of a lower bound, we may ignore local factors of ramified primes.
\end{proof}
\begin{prop}\label{prop:estimate_on_sum_of_F'(n)Gx(n)}
Assuming GRH, for all $\epsilon>0$, $$\sum_{n=1}^X F'(n)G_x(n)=X\Res_{s=1}L_{x,lower}(s)+O_{\epsilon}(X^{1/2+\epsilon}x^{\epsilon})$$
\end{prop}
\begin{proof}
    This follows the same proof as Proposition \ref{prop:estimate_on_sum_of_F(n)Gx(n)}, noting that for the sake of a lower bound, we may ignore local factors of ramified primes.
\end{proof}
\begin{cor}
\label{cor:lower_bound_Tx}
Assuming GRH, $$|T_x|=\Omega_{\epsilon}(B\frac{L(\chi_x,1)}{x}+B^{1/2+\epsilon}x^{-1/4+\epsilon}).$$
\end{cor}
\begin{proof}
    Combine Lemma \ref{lem:Lxlowers_residue} and Proposition \ref{prop:estimate_on_sum_of_F'(n)Gx(n)}, and apply it to the case where $X = \frac{B}{\sqrt{x}}$.
\end{proof}
From the Corollary, we get

\begin{prop}
\label{prop:thm_main_lower_bound}
$|T|=\Omega(B \log B)$.
\end{prop}
\begin{proof}
Since $\bigsqcup_{2 \leq x \leq B^{2/3-\epsilon}} T_x \subset T$, we have that \begin{align*}
    |T| &\geq \sum_{\substack{2 \leq x\leq B^{2/3-\epsilon} \\ \text{ even, squarefree}}} |T_x|\\
    &=\sum_{\substack{2 \leq x\leq B^{2/3-\epsilon} \\ \text{ even, squarefree}}} \Omega_{\delta}( B\frac{L(\chi_x,1)}{x}+B^{1/2+\delta}x^{-1/4+\delta}).
\end{align*}
Taking $\delta=\frac{9}{20}\epsilon$, the second term becomes $O(B)$. For the first term, we do summation by parts, and obtain \begin{equation}
\label{quickref1}
    B\left(O(1)+\sum_{k \leq B^{2/3-\epsilon}}\frac{\sum_{\substack{x\leq k \\ \text{ even, squarefree}}}L(\chi_x,1)}{k(k+1)}\right).
\end{equation} So, it suffices to show $$\sum_{\substack{x\leq k \\ \text{ even, squarefree}}} L(\chi_x,1)$$ is at least on the order of $k$. The proof of this is almost exactly the same as the proof outlined in \cite[Theorem V.5.3]{ayoub_2006}, and so we omit it.

\end{proof}

\section{Appendix A}
The following appendix almost follows word for word \cite[Appendix A]{chambert-loir_tschinke_2001}. The only difference is that we kept track of the error dependence on the L-function.

\begin{theorem}
\label{thm:appendix_wiener_ikehara_with_error}
Let $(\lambda_n)$ be an increasing sequence of positive real numbers. Let $(c_n)$ be a sequence of nonnegative real numbers. Define $$f(s)=\sum_{n=0}^\infty \frac{c_n}{\lambda_n^s}.$$ We make the following hypotheses \begin{enumerate}
    \item $f$ converges on the half-plane $\Re(s) > a > 0$;
    \item $f$ admits a meromorphic continuation to $\Re(s)>a-\delta_0 > 0$; 
    \item and on this domain, there is a unique pole at $s=a$, of multiplicity $b \in \bbN$. We note that $\Theta = \lim_{s \to a} f(s)(s-a)^b > 0$.
    \item Finally, we assume there exists positive constants $r,\kappa$ so that $$|f(s)\frac{(s-a)^b}{s^b}| \leq r (1+\Im(s))^\kappa.$$
\end{enumerate}

Then, there exists a monic polynomial $P$ of degree $b-1$ so that for $\delta\leq \delta_0$, as $X$ tends to $\infty$, we have $$N(X):= \sum_{\lambda_n \leq X}c_n = \frac{\Theta}{a(b-1)!}X^a P(\log(X))+O(rX^{a-\delta}).$$

\end{theorem}

We define, for positive integers $k$, the function $$\phi_k(X)=\sum_{\lambda_n \leq X}a_n(\log(X/\lambda_n))^k,$$ and define $\phi_0(X)=N$. 

\begin{lemma}
With the hypothesis of the previous theorem, there exists for integers $k > \kappa$ a polynomial $Q$ of degree $b-1$ with leading coefficient $\frac{k! \Theta}{a^{k+1}(b-1)!}$ so that for $\delta < \delta_0$, we have $$\phi_k(X)=X^a Q( \log X) + O(rX^{a-\delta}),$$ where the implied constant is independent of $f$. 
\end{lemma}
\begin{proof}
For any $a' > a$, we have $$\int_{a'+i\bbR}\lambda^s \frac{ds}{s^{k+1}} = \frac{2\pi i}{k!}(\log^+(\lambda))^k, \lambda > 0.$$ Consequently, $$\phi_k(X)=\frac{k!}{2 \pi i}\int_{a' + i\bbR}f(s)X^s \frac{ds}{s^{k+1}},$$ where the integral is absolutely convergent as $k>\kappa$.

We now shift the contour of integration to $\Re(s)=a-\delta$. We pick up a pole at $s=a$ with residue $\frac{\Theta}{a^{k+1}(b-1)!}X^a Q(\log X)$, where $Q$ is a monic polynomial of degree $b-1$. The resulting integral $$\int_{\Re(s)=a-\delta} f(s)X^s \frac{ds}{s^{k+1}}$$ is an error term, and here we use the bound $|f(s) \frac{(s-a)^b}{s^b}| \leq r(1+\Im(s))^{\kappa}$ and obtain that the integral is $O(rX^{a-\delta})$.
\end{proof}

Now we may prove Theorem \ref{thm:appendix_wiener_ikehara_with_error}.
\begin{proof}
We prove by inductive descent on $k$ that the previous lemma is in fact true for all integers $k$. For $k=0$, the theorem is trivially proven. Now assuming the previous lemma holds for some fixed $k\geq 1$, we show that it is in fact true for $k-1$.

For any $\eta \in (0,1)$, we have the inequality $$\frac{\phi_k(X(1-\eta))-\phi_k(X)}{\log(1-\eta)} \leq k \phi_{k-1}(X) \leq \frac{\phi_k(X(1+\eta))-\phi_k(X)}{\log(1+\eta)}.$$ Now fix $\delta'$ with $0 < \delta < \delta' < \delta_0$. By the previous lemma, $$|\phi_k(X)-X^a Q( \log X)| = O(rX^{a-\delta'}).$$ Consequently, for $-1 < u < 1$,  $$\frac{\phi_k(X(1+u))-\phi_k(X)}{\log(1+u)}=\frac{k! \Theta}{a^{k+1}(b-1)!}X^a \frac{Q(\log X + \log (1+u))(1+u)^a-Q(\log X)}{\log(1+u)}+R(X),$$ where $$|R(X)|=O(\frac{2rX^{a-\delta'}}{|\log(1+u)|})=O(r\frac{X^{a-\delta'}}{u}).$$ Now we have a series of equalities. \begin{align*}
    &\frac{Q(\log X + \log (1+u))(1+u)^a-Q(\log X)}{\log(1+u)}\\
    &=Q(\log X)\frac{(1+u)^a -1}{\log(1+u)}+\sum_{n=1}^{b-1}\frac{1}{n!}Q^{(n)}(\log(X))(\log(1+u)^{n-1}(1+u)^a)\\
    &=Q(\log X)(a+O(u))+Q'(\log(X))(1+O(u))+O((\log X)^{b-1} u)\\
    &=(aQ+Q')(\log X)+O((\log X)^{b-1}u).
\end{align*} Take $u=\pm 1/X^{\epsilon}$, where $\delta'-\epsilon > \delta$. Then, we have $|R(X)|=O(rX^{a-\delta})$ and $$\frac{Q(\log X + \log (1+u))(1+u)^a-Q(\log X)}{\log(1+u)}=(aQ+Q')(\log X)+O(rX^{-\delta}).$$ Thus, we have that $\phi_{k-1}(X)=\frac{1}{k} X^a (aQ+Q')(\log X)+O(rX^{a-\delta})$ and the leading coefficient of $\frac{1}{k}(aQ+Q')$ is precisely $\frac{(k-1)!\Theta}{a^k (b-1)!}$.
\end{proof}

\printbibliography

\end{document}